\documentclass[12pt,a4paper]{amsart}
\usepackage[utf8]{inputenc}
\usepackage[english]{babel}
\usepackage[T1]{fontenc}
\usepackage{amsmath}
\usepackage{amsfonts}
\usepackage{amssymb}
\usepackage{vmargin}
\author{Léo Morin}
\title[A Semiclassical Birkhoff Normal Form]{A Semiclassical Birkhoff Normal Form for symplectic magnetic wells}
\address{Univ Rennes, CNRS, IRMAR - UMR 6625, F-35000 Rennes, France}

\makeatletter

\@addtoreset{equation}{section}
\makeatother

\newtheorem*{theorem*}{Theorem}
\newtheorem{theorem}{Theorem}[section]
\newtheorem{lemma}{Lemma}[section]

\newtheorem{corollary}{Corollary}[section]
\newtheorem{proposition}{Proposition}[section]
\newtheorem{assumption}{Assumption}

\newcommand{\A}{\mathbf{A}}
\newcommand{\B}{\mathbf{B}}
\newcommand{\R}{\mathbf{R}}

\newcommand{\N}{\mathbf{N}}
\newcommand{\Z}{\mathbf{Z}}

\newcommand{\dd}{\mathrm{d}}

\newcommand{\zz}{\hat{z}}
\newcommand{\bbeta}{\hat{\beta}}

\newcommand{\grandO}{\mathcal{O}}

\newcommand{\supp}{\mathsf{supp}}
\newcommand{\spectre}{\mathsf{sp}}
\newcommand{\Hess}{\mathsf{Hess}}

\newcommand{\Op}{\mathsf{Op}^w_{\hbar}}

\newcommand{\weylsymbol}{\sigma_{\hbar}^{w}}
\newcommand{\symbolseries}{\sigma_{\hbar}^{w,\mathsf{T}}}

\newcommand{\psit}{\tilde{\psi}}

\newcommand{\Ld}{\mathsf{L}^{2}}

\newcommand{\Bh}{\mathcal{B}_{\hbar}}

\newcommand{\Uh}{U_{\hbar}}
\newcommand{\Uht}{\tilde{U}_{\hbar}}
\newcommand{\Rh}{R_{\hbar}}
\newcommand{\Ih}{\mathcal{I}_{\hbar}}

\newcommand{\Lh}{\mathcal{L}_{\hbar}}
\newcommand{\Lhc}{\widehat{\mathcal{L}}_{\hbar}}

\newcommand{\BNh}{\mathcal{N}_{\hbar}}
\newcommand{\ad}{\mathsf{ad}}

\begin{document}

\maketitle

\begin{abstract}
In this paper we construct a Birkhoff normal form for a semiclassical magnetic Schrödinger operator with non-degenerate magnetic field, and discrete magnetic well, defined on an even dimensional riemannian manifold $M$. We use this normal form to get an expansion of the first eigenvalues in powers of $\hbar^{1/2}$, and semiclassical Weyl asymptotics for this operator.
\end{abstract}

\section{\textbf{Introduction}}

The analysis of the magnetic Schrödinger operator, or magnetic Laplacian, on a Riemannian manifold $$\Lh = (i\hbar \dd + A)^* (i\hbar \dd + A)$$ in the semiclassical limit $\hbar \rightarrow 0$ has given rise to many investigations in the last twenty years. Asymptotic expansions of the lowest eigenvalues have been studied in many cases involving the geometry of the possible boundary of $M$ and the variations of the magnetic field. For discussions about the subject, the reader is referred to the books and review \cite{FoHe10}, \cite{HeKo14}, \cite{livre}. The classical picture associated with the Hamiltonian $$\vert p - A(q) \vert^2$$ has started being investigated to describe the semiclassical bound states (the eigenfunctions of low energy) of $\Lh$, in \cite{article} (on $\R^2$) and \cite{3D} (on $\R^3$). In these two papers, semiclassical Birkhoff normal forms were used to describe the first eigenvalues. In \cite{Sjo92}, Sjöstrand introduced the semiclassical Birkhoff normal form to study the spectrum of an electric Schrödinger operator, and some resonance phenomenons appeared. In \cite{San}, the resonant case for the same electric Schrödinger operator was tackled (see also \cite{San06} and \cite{San09}). In this paper, we adapt this method to $\Lh$, following the ideas of \cite{article}. Some normal forms for magnetic Schrödinger operators also appear in \cite{Ivrii}. On a Riemannian manifold $M$, the magnetic Schrödinger operator is related to the Bochner Laplacian (see the recent papers \cite{Ko18} and \cite{MaSa18}, where bounds and asymptotic expansions of the first eigenvalues of Bochner Laplacians are given).\\

In this paper we get an expansion of the first eigenvalues of $\Lh$ in powers of $\hbar^{1/2}$, and semiclassical Weyl asymptotics. It would be interesting to have a precise description of the eigenfunctions too, as was done in the 2D case by Bonthonneau-Raymond \cite{BonRa17} (euclidian case) and Nguyen Duc Tho \cite{Tho} (general riemannian metric). Moreover, we only have investigated the spectral theory of the stationary Schrödinger equation with a pure magnetic field ; it would be interesting to describe the long-time dynamics of the full Schrödinger evolution, as was done in the euclidian 2D case by Boil-Vu Ngoc \cite{BoiVu}.

\subsection{Definition of the magnetic Schrödinger operator} 

Let $(M,g)$ be a smooth $d$ dimensional oriented Riemannian manifold, either without boundary or with smooth boundary. In particular we can take $M= \R^d$ with the Euclidean metric, or $M$ compact with boundary.
For $q \in M$, $g_q$ is a scalar product on $T_qM$. Since $M$ is oriented, there is a canonical volume form, denoted either $\dd x_g$ or $\dd q_g$. If $f \in \Ld (M)$, we denote its norm by $$\Vert f \Vert =  \left( \int_M \vert f(q) \vert^2 \dd q_g \right)^{1/2}.$$ If $p \in T_qM^*$, we denote by $\vert p \vert_{g^{\star}_q}$ or $\vert p \vert$ the norm of $p$, defined by
\begin{align}\label{defgstar}
\forall Q \in T_qM, \quad \vert Q \vert^2_{g_q} = \vert g_q(Q,.) \vert^2_{g^*_q}.
\end{align}
We denote by $g_q^*$ the associated scalar product. The norm of a $1$-form $\alpha$ on $M$ is 
$$\Vert \alpha \Vert = \left( \int_M \vert \alpha(q) \vert^2 \dd {q_g} \right)^{1/2}.$$
It is associated with a scalar product, denoted by brackets $\langle . , . \rangle$. 

We denote by $\dd$ the exterior derivative, associating to any $p$-form $\alpha$ a $(p+1)$-form $\dd \alpha$. Using the scalar products induced by the metric, we can define its adjoint $\dd^*$, associating to any $p$-form $\alpha$ a $(p-1)$-form $\dd^* \alpha$.\\

We take a $1$-form $A$ on $M$ called the magnetic potential, and we denote by $B = \dd A$ its exterior derivative. $B$ is called the magnetic $2$-form. The associated classical Hamiltonian is defined on $T^*M$ by:
$$H(q,p) = \vert p - A(q) \vert^2_{g^*_q}, \quad p \in T_qM^*.$$

Using the isomorphism $T_qM \simeq T_qM^*$ given by the metric, we define the magnetic operator $\B(q):T_qM \rightarrow T_qM$ by:
\begin{align}\label{defBmatrix}
B_q(Q_1,Q_2) = g_q(\B(q)Q_1,Q_2),\quad \forall Q_1, Q_2 \in T_qM.
\end{align}
The norm of $\B(q)$ is
$$\vert \B(q) \vert = [ \text{Tr}( \B^*(q) \B(q) ) ]^{1/2}.$$

On the quantum side, for $\hbar>0$, we define the magnetic quadratic form $q_{\hbar}$ on
$$D(q_{\hbar}) = \lbrace u \in \Ld (M), (i\hbar d + A)u \in \Ld \Omega^1(M), u_{\partial M} = 0 \rbrace,$$
by
$$q_{\hbar}(u) = \int_M \vert (i \hbar d + A )u \vert^2 \dd q_g,$$
where $\Ld \Omega^1(M)$ denotes the space of square-integrable $1$-forms on $M$. By the Lax-Milgram theorem, this quadratic form defines a self-adjoint operator $\Lh$ on $$D(\Lh) = \lbrace u \in \Ld (M), (i\hbar \dd + A)^*(i\hbar \dd + A) u \in \Ld (M), u_{\partial M} = 0 \rbrace,$$
by the formula
$$\langle \Lh u, v \rangle = q_{\hbar}[u,v], \quad \forall u,v \in \mathcal{C}^{\infty}_0(M),$$
where $q_{\hbar}[.,.]$ is the inner product associated with the quadratic form $q_{\hbar}(.)$. $\Lh$ is the magnetic Schrödinger operator with Dirichlet boundary conditions.

\subsection{Local coordinates} \label{sectioncoordinates}

If we choose local coordinates $q=(q_1,...,q_d)$ on $M$, we get the corresponding vector fields basis $(\partial_{q_1}, ..., \partial_{q_d})$ on $T_qM$, and the dual basis $(\dd q_1 , ..., \dd q_d)$ on $T_qM^*$. In these basis, $g_q$ can be identified with a symmetric matrix $(g_{ij}(q))$ with determinant $\vert g \vert$, and $g^*_q$ is associated with the inverse matrix $(g^{ij}(q))$. We can write the $1$-form $A$ in the coordinates:
$$A \equiv A_1 \dd q_1 + ... + A_d \dd q_d,$$
with $\A = (A_j)_{1 \leq j \leq d} \in \mathcal{C}^{\infty}(\R^d,\R^d)$. We denote 
$$T_qA : T_qM \rightarrow T_qM^*$$
the linear operator whose matrix is the Jacobian of $\A$: $$(\nabla \A(q))_{ij} = \partial_j A_i(q).$$
In the coordinates, the $2$-form $B$ is
$$B= \sum_{i<j} B_{ij} \dd q_i \wedge \dd q_j,$$
with
\begin{align}\label{BenCoordonnees}
B_{ij} = \partial_i A_j - \partial_j A_i = (^t \nabla \A - \nabla \A)_{ij}.
\end{align}
Let us denote $(\B_{ij}(q))_{1 \leq i,j \leq d}$ the matrix of the operator $\B(q) : T_qM \rightarrow T_qM$ in the basis $(\partial q_1 , ..., \partial q_d)$. With this notation, equation (\ref{defBmatrix}) relating $\B$ to $B$ can be rewritten:
$$\forall Q,\tilde{Q} \in \R^d, \quad \sum_{ijk} g_{kj} \B_{ki} Q_i \tilde{Q}_{j} = \sum_{ij} B_{ij} Q_i \tilde{Q}_j,$$
which means that
\begin{align} \label{relationBBmatrice}
\forall i,j, \quad B_{ij} = \sum_k g_{kj} \B_{ki}.
\end{align}
Also note that:
\begin{small}
\begin{align}\label{iQBCoordonees}
\iota_Q B &= \sum_{i<j} B_{ij} \left( Q_i \dd q_j - Q_j \dd q_i \right) = \sum_j \left( \sum_i B_{ij}Q_i \right) \dd q_j \\ &=  \sum_j \left[ (\ ^t \nabla \A - \nabla \A ) Q \right]_j \dd q_j = (\ ^t T_qA - T_qA ) Q
\end{align}
\end{small}
Finally, in the coordinates $H$ is given by:
\begin{align}\label{HamiltonienCoordonnees}
H(q,p) = \sum_{i,j} g^{ij}(q) (p_i - A_i(q))(p_j - A_j(q)),
\end{align}
and $\Lh$ acts as the differential operator:
\begin{align}\label{Lhcoord}
\Lh^{\text{coord}} = \sum_{k,l =1}^d \vert g \vert^{-1/2} (i\hbar \partial_k + A_k) g^{kl} \vert g \vert^{1/2} (i\hbar \partial_l+A_l).
\end{align}

\subsection{Pseudodifferential operators}

We refer to \cite{Martinez} and \cite{Zworski} for the general theory of $\hbar$-pseudodifferential operators. If $m\in \Z$, we denote by $$S^m(\R^{2n})= \lbrace a \in \mathcal{C}^{\infty}(\R^{2n}), \vert \partial^{\alpha}_x \partial^{\beta}_{\xi} a \vert \leq C_{\alpha \beta} \langle \xi \rangle ^{m-\vert \beta \vert}, \quad \forall \alpha, \beta \in \N^d \rbrace$$ the class of Kohn-Nirenberg symbols. If $a$ depends on the semiclassical parameter $\hbar$, we require that the coefficients $C_{\alpha \beta}$ are uniform with respect to $\hbar \in (0, \hbar_0]$. For $a_{\hbar} \in S^m(\R^{2n})$, we define its associated Weyl quantization $\Op(a_{\hbar})$ by the oscillatory integral
$$\mathcal{A}_{\hbar}u(x) = \Op (a_{\hbar})u(x)= \frac{1}{(2\pi \hbar)^n} \int _{\R^{2n}} e^{\frac{i}{\hbar} \langle x-y, \xi \rangle } a_{\hbar}\left(\frac{x+y}{2},\xi \right) u(y)\dd y \dd \xi,$$ and we denote:
$$a_{\hbar} = \sigma_{\hbar}(\mathcal{A}_{\hbar}).$$
A pseudodifferential operator $\mathcal{A}_{\hbar}$ on $\Ld(M)$ is an operator acting as a pseudodifferential operator in coordinates. Then the principal symbol of $\mathcal{A}_{\hbar}$ does not depend on the coordinates, and we denote it by $\sigma_0(\mathcal{A}_{\hbar}).$ The subprincipal symbol $\sigma_1(\mathcal{A}_{\hbar})$ is also well-defined, up to imposing the charts to be volume-preserving (in other words, if we see $\mathcal{A}_{\hbar}$ as acting on half-densities, its subprincipal symbol is well defined). 

In any local coordinates, the coefficients $A_j$ of $A$ (as a function of $q \in \R^d$) are in $S^1(\R^{2d}_{(q,p)})$. Hence we see from (\ref{Lhcoord}) that $\Lh$ is a pseudodifferential operator on $\Ld(M)$. Its principal and subprincipal Weyl symbols are:
$$\sigma_0(\Lh) = H, \quad \sigma_1(\Lh) = 0.$$
This is well-known, but we detail the computation of the subprincipal symbol in Appendix (Lemma \ref{SubprincipalSymbol}).

\subsection{Assumptions}
Since $\B(q)$, defined in (\ref{defBmatrix}), is a skew-symmetric operator for the scalar product $g_q$, its eigenvalues are in $i \R$.
We define the magnetic intensity, which is equivalent to the trace-norm, by
$$b(q) = \text{Tr}^+ \B(q) = \frac{1}{2}\text{Tr}( [ \B^*(q) \B(q) ]^{1/2}) = \sum_{i \beta_j \in \spectre(\B(q)), \beta_j >0} \beta_j.$$
It is a continuous function of $q$, but not smooth in general. We also denote
$$b_0 = \inf_{q \in M} b(q),$$
$$b_{\infty} = \liminf_{\vert q \vert \rightarrow + \infty} b(q).$$

We first assume that the magnetic field satisfies the following inequality.

\begin{assumption}\label{InequalityIntensity}
We assume that there exist $\hbar_0 >0$ and $C_0>0$ such that, for $\hbar \in (0, \hbar_0]$,
$$\forall u \in D(q_{\hbar}), \quad (1+\hbar^{1/4}C_0) q_{\hbar}(u) \geq \int_{M} \hbar (b(q) - \hbar^{1/4} C_0) \vert u(q) \vert^2 \dd q_g.$$
\end{assumption}

In the Appendix (Lemma \ref{manifoldinequality}), we describe cases when Assumption \ref{InequalityIntensity} holds. In particular, it holds if $M$ is compact. If $M=\R^d$, it is true if we assume that
$$\Vert \nabla \B_{ij}(q) \Vert \leq C (1 + \vert \B(q) \vert),$$
for some $C>0$. These results are adapted from \cite{HeMo96}. \\

We consider the case of a unique discrete magnetic well:

\begin{assumption}\label{dismagwell}
We assume that the magnetic intensity $b$ admits a unique and non-degenerate minimum $b_0$ at $q_0 \in M \setminus \partial M$, such that $0 < b_0 < b_{\infty}$.
\end{assumption}

Finally, we make a non-degeneracy assumption.

\begin{assumption}\label{nondegeneracy}
We assume that $d$ is even and $\B(q_0)$ is invertible. 
\end{assumption}

In particular, $\B(q)$ is invertible for $q$ in a neighborhood of $q_0$, which means that the $2$-form $B$ is symplectic near $q_0$.
Under this Assumption, the eigenvalues of $\B(q_0)$ can be written
$$ \pm i \beta_1(q_0), \ldots , \pm i \beta_{d/2}(q_0),$$ with $\beta_j(q_0) >0$. We define the resonance order $r_0 \in \N^* \cup \lbrace \infty \rbrace$ of the eigenvalues by
\begin{align}\label{defrzero}
r_0 := \min \lbrace  \vert \alpha \vert : \alpha \in \Z^{d/2}, \alpha \neq 0, \langle \alpha , \beta(q_0) \rangle = 0 \rbrace,
\end{align}
with the notation
$$\langle \alpha, \beta(q_0) \rangle := \sum_{j=1}^{d/2} \alpha_j \beta_j(q_0).$$
We make a non-resonance assumption.
\begin{assumption} We assume that the eigenvalues of $\B(q_0)$ are simple (which is equivalent to assuming that $r_0 \geq 3$).
\end{assumption}
In particular, there is a neighborhood $\Omega \subset \subset M \setminus \partial M$ of $q_0$ on which the eigenvalues of $\B(q)$ are simple, and defined by smooth positive functions
$$\beta_j : \Omega \rightarrow \R_+^*.$$
We can choose $\Omega$ such that every $\beta_j$ is bounded from bellow by a positive constant on $\Omega$. We can also find smooth orthonormal vectors on $\Omega$: $$u_1(q),v_1(q), \ldots , u_{d/2}(q), v_{d/2}(q) \in T_q M,$$
such that: 
\begin{align}\label{defujvj}
\B(q) u_j(q) = - \beta_j(q) v_j(q), \quad \B(q) v_j(q) = \beta_j(q) u_j(q).
\end{align}
\linebreak
We take 
\begin{align}\label{defr}
r \in \N \cap [3, r_0].
\end{align} Up to reducing $\Omega$ (depending on $r$), we also have (since $r$ is finite), for $0 < \vert \alpha \vert < r$:
\begin{align}\label{fffffrrr}
\langle \alpha , \beta (q) \rangle \neq 0, \quad \forall q \in \Omega.
\end{align}
Under Assumption \ref{dismagwell}, we can find $b_0 < \tilde{b}_1< b_{\infty}$ such that
\begin{align}\label{defb1tilde}
K:= \lbrace b(q) \leq \tilde{b}_1 \rbrace \subset \Omega.
\end{align}
Using the inequality in Assumption 1, it is proved in \cite{HeMo96} that there exist $\hbar_0$ and $c>0$ such that, for $\hbar \in (0, \hbar_0]$, 
$$\spectre_{ess} (\Lh) \subset [\hbar (\tilde{b}_1-c \hbar^{1/4}), + \infty ),$$
and so, for $\hbar$ small enough, the spectrum of $\Lh$ below $\hbar b_1$ (for a given $b_1 < \tilde{b}_1$) is discrete.\\

\subsection{Main results}

On the classical part, we first prove the following reduction of the Hamiltonian. For $z=(x,\xi) \in \R^d$, we denote $z_j=(x_j,\xi_j)$ and $B_z(\varepsilon) = \lbrace \vert z \vert \leq \varepsilon \rbrace.$

\begin{theorem}\label{symplectomorphism} Under Assumptions 1,2,3 and 4, for $\Omega$ and $\varepsilon >0$ small enough, there exist symplectomorphisms
$$\varphi : (\Omega,B) \rightarrow (V \subset \R_w^d, \dd \eta \wedge \dd y),$$
and
$$ \Phi : \left( V \times B_z(\varepsilon), \dd \eta \wedge \dd y + \dd \xi \wedge \dd x \right) \rightarrow (U \subset T^*M, \omega), $$ with $\Phi(\varphi(q),0) = (q,A(q))$, under which the Hamiltonian $H$ becomes:
$$\hat{H}(w,z) = H \circ \Phi (w,z) = \sum_{j=1}^{d/2} \hat{\beta}_j(w) \vert z_j \vert^2 + \grandO (\vert z \vert^3),$$
locally uniformly in $w$, with the notation $\hat{\beta}_j(w) = \beta_j \circ \varphi^{-1}(w)$.
\end{theorem}

Our next aim is to construct a semiclassical Birkhoff normal form for $\Lh$, that is to say a pseudodifferential operator $\BNh$ on $\Ld(\R^d)$, commuting with suitable harmonic oscillators such that:
$$\Uh \Lh \Uh^* = \BNh + \Rh,$$
with $\Uh : \Ld (M) \rightarrow \Ld(\R^d)$ a microlocally unitary Fourier integral operator and $\Rh$ a remainder. We will contruct the remainder so that the first eigenvalues of $\Lh$ coincide with the first eigenvalues of $\BNh$, up to a small error of order $\grandO(\hbar^{r/2 - \varepsilon})$, where $r$ is defined in (\ref{defr}). More precisely, we prove the following theorem.

\begin{theorem}[Semiclassical Birkhoff normal form]\label{ThmFormeNormale}
We denote by $z=(x,\xi)\in T^*\R^{d/2}_x$ and $w=(y,\eta)\in T^* \R^{d/2}_y$ the canonical variables. For $\zeta >0$ and $\hbar \in (0,\hbar_0]$ small enough, there exist a Fourier integral operator $$\Uh : \Ld (\R^d _{(x,y)}) \rightarrow \Ld (M),$$ a smooth function $f^{\star}(w,I_1,...,I_{d/2},\hbar)$, and a pseudodifferential operator $\Rh$ on $\R^d$ such that:
\begin{align*}
&(i)\quad \Uh^* \Lh \Uh = \Lh^0 + \Op f^{\star}(w, \Ih^{(1)}, ..., \Ih^{(d/2)}, \hbar) + \Rh,\\
&(ii)\quad (1-\zeta) \langle \Lh^0 \psi, \psi \rangle \leq \langle \BNh \psi , \psi \rangle \leq (1+\zeta) \langle \Lh^0 \psi, \psi \rangle, \quad \forall \psi \in \mathcal{S}(\R^d),\\
&(iii)\quad \weylsymbol (\Rh) \in \grandO((\vert z \vert + \hbar^{1/2})^{r}) \text{ on a neighborhood of } w=0,\\
&(iv) \quad  \Uh^* \Uh = I \text{ microlocally near } (z,w)=0, \\
&(v) \quad \Uh \Uh^* = I \text{ microlocally near } (q,p) = (q_0, A_{q_0}),
\end{align*}
with 
\begin{align}\label{DefOscillateurs}
\Ih^{(j)} = \Op( \vert z_j \vert^2 ) = -\hbar^2 \frac{\partial^2}{\partial x_j^2} + x_j^2,\quad \Lh^0 = \Op \left( \sum_{j=1}^{d/2} \hat{\beta}_j(w) \vert z_j \vert^2 \right).
\end{align}
We call $$\BNh = \Lh^0 + \Op f^{\star}(w, \Ih^{(1)}, ..., \Ih^{(d/2)}, \hbar)$$
the normal form, and $\Rh$ the remainder.
\end{theorem}

Using microlocalization properties of the eigenfunctions of $\Lh$ and $\BNh$, we prove that they have the same spectra in the following sense. We recall that $\tilde{b}_1$, defined in (\ref{defb1tilde}), is chosen such that
$$\lbrace b(q) \leq \tilde{b}_1 \rbrace \subset \Omega.$$

 \begin{theorem}\label{comparisonOftheEigenvalues}
Let $\varepsilon >0$ and $b_1 \in (0,\tilde{b}_1)$. We denote $$\lambda_1(\hbar) \leq \lambda_2(\hbar) \leq ... \quad \text{and} \quad \nu_1(\hbar) \leq \nu_2(\hbar) \leq ...$$ the first eigenvalues of $\Lh$ and $\BNh$ respectively. Then
$$\lambda_n(\hbar) = \nu_n(\hbar) + \grandO(\hbar^{r/2 - \varepsilon}),$$
uniformly in $n$ such that $\lambda_n(\hbar) \leq \hbar b_1$ and $\nu_n(\hbar) \leq \hbar b_1$.
\end{theorem}

We also reduce $\BNh$ according to harmonic oscillators.

\begin{theorem}\label{ThmReductionDeBNh}
For $k \geq 0$, let us denote $h_{k}$ the Hermite function, satisfying
$$\Ih^{(j)} h_{k}(x_j) = \hbar (2k +1) h_k(x_j).$$
For $n=(n_1,...,n_{d/2}) \in \N^{d/2}$, there exists a pseudodifferential operator $\BNh^{(n)}$ acting on $\Ld (\R_y^{d/2})$ such that:
$$\BNh(u \otimes h_{n_1} \otimes ... \otimes h_{n_{d/2}} ) = \BNh^{(n)}(u) \otimes h_{n_1} \otimes ... \otimes h_{n_{d/2}}, \quad u \in \mathcal{S}(\R^{d/2}_y).$$
Its symbol is:
$$F^{(n)}(w) = \hbar \sum_{j=1}^{d/2} \bbeta_j(w)(2n_j+1) + f^{\star}(w,\hbar(2n+1),\hbar),$$
and we have:
$$\spectre (\BNh) = \bigcup_n \spectre (\BNh^{(n)}).$$
Moreover, the multiplicity of $\lambda$ as eigenvalue of $\BNh$ is the sum over $n$ of the multiplicities of $\lambda$ as eigenvalue of $\BNh^{(n)}$.
\end{theorem} 
 
Finally, we deduce an expansion of the $N>0$ first eigenvalues of $\Lh$ in powers of $\hbar^{1/2}$.
 
\begin{theorem}[Expansion of the first eigenvalues]\label{expansionOftheEigenvalues}
Let $\varepsilon >0$ and $N \geq 1$. There exist $\hbar_0 >0$ and $c_0 >0$ such that, for $\hbar \in (0,\hbar_0]$, the $N$ first eigenvalues of $\Lh$ : $(\lambda_j(\hbar))_{1 \leq j \leq N}$ admit an expansion in powers of $\hbar^{1/2}$ of the form:
$$\lambda_j(\hbar) = \hbar b_0 + \hbar^2(E_j+c_0) + \hbar^{5/2} c_{j,5} + ... + \hbar^{ (r-1)/2 } c_{j,r-1} + \grandO(\hbar^{r/2 - \varepsilon}),$$
where $\hbar E_j$ is the $j$-th eigenvalue of the $d/2$-dimensional harmonic oscillator $$\Op( \Hess_0 (b \circ \varphi^{-1}) ).$$
\end{theorem}

Note that, from Theorems \ref{comparisonOftheEigenvalues} and \ref{ThmReductionDeBNh}, we deduce Weyl estimates for $\Lh$. Some similar formulas appear in \cite{Ivrii}. Here $N(\Lh, b_1 \hbar)$ denotes the number of eigenvalues 
$\lambda$ of $\Lh$ such that $\lambda \leq b_1 \hbar$, counted with multiplicities.

\begin{corollary}[Weyl estimates]\label{WeylEstimates}
For any $b_1 \in (b_0, \tilde{b}_1)$,
\begin{align*}
N(\Lh, b_1 \hbar) \sim \frac{1}{(2\pi \hbar)^{d/2}} \sum_{n \in \N^{d/2}} \int_{b^{[n]}(q) \leq b_1} \frac{B^{d/2}}{(d/2)!}.
\end{align*}
where $$b^{[n]}(q) = \sum_{j=1}^{d/2} (2n_j+1) \beta_j(q).$$
The sum is finite because the $\beta_j$ are bounded from below by a positive constant on $\Omega$. In particular, if $M= \R^d$, we get
$$N(\Lh,b_1 \hbar) \sim \frac{1}{(2\pi \hbar)^{d/2}} \sum_{n \in \N^{d/2}} \int_{b^{[n]}(q) \leq b_1} \beta_1(q) ... \beta_{d/2}(q) \dd q.$$
\end{corollary}

\subsection{Organization and strategy}

In section 2, we construct a symplectomorphism which simplify $H$ near its zero set $\Sigma = H^{-1}(0)$ (Theorem \ref{symplectomorphism}). In the new coordinates, $H$ becomes:
$$\hat{H}(q,z) = \sum_{j=1}^{d/2} \beta_j(q) \vert z_j \vert^2 + \grandO(\vert z \vert^3).$$ In section 3, we construct a formal Birkhoff normal form: in the space of formal series in variables $(x,\xi, \hbar)$, we change $\hat{H}$ into $H^0 + \kappa + \rho$, with $H^0 = \sum_{j=1}^{d/2} \beta_j \vert z_j \vert^2$, $\kappa$ a series in $\vert z_j \vert^2$ ($1 \leq j \leq d/2$), and $\rho$ a remainder of order $r$ (Theorem \ref{algtheorem}). In section 4, we quantify the changes of coordinates constructed in section 2 and 3, and we get the semiclassical Birkhoff normal form (Theorem \ref{ThmFormeNormale}). In section 5, we reduce $\BNh$ (Theorem \ref{ThmReductionDeBNh}) and we deduce an expansion of its first eigenvalues. It remains prove that the spectra of $\Lh$ and $\BNh$ below $b_1 \hbar$ coincide. Before doing it, we need microlocalization results proved in section 6. We prove that the eigenfunctions of $\Lh$ and $\BNh$ are microlocalized near the zero set of $H$, where our formal construction is valid. In section 7, we use the results of section 6, to prove that $\Lh$ and $\BNh$ have the same spectrum below $b_1\hbar$ (Theorem \ref{comparisonOftheEigenvalues}). This Theorem, together with the results of section 5, finishes the proof of Theorem \ref{expansionOftheEigenvalues}. We also prove the Weyl estimates (Corollary 1.1) here. Finally, in section 8 we discuss what we can get in the case $r_0=\infty$.

\section{\textbf{Reduction of the classical Hamiltonian}}

\subsection{A symplectic reduction of $T^*M$}

The zero set of $H$:
$$\Sigma = \lbrace (q,A(q)) \in T^*M : q \in \Omega \rbrace,$$
is a $d$-dimensional smooth submanifold of the cotangent bundle $T^*M$. We denote $j :\Omega \rightarrow T^*M$ the embedding
$$j(q) = (q,A(q)).$$
The symplectic structure on $T^*M$ is defined by the form $$\omega = \dd p \wedge \dd q = \dd \alpha, \quad \alpha = p \dd q.$$ In other words, for $p \in T_qM^*$ and $ \mathcal{V} \in T_{(q,p)}(T^*M),$
\begin{align}\label{defLiouville}
\alpha_{(q,p)} (\mathcal{V}) = p(\pi_* \mathcal{V}),
\end{align}
Where the map $\pi_* : T_{(q,p)}(T^*M) \rightarrow T_qM$ is the differential of the canonical projection $$\pi : T^*M \rightarrow M, \quad \pi(q,p) =q.$$
Using local coordinates with the notations of section \ref{sectioncoordinates}, at any point $(q,p) \in T^*M$ with 
$$p = p_1 \dd q_1 + ... + p_d \dd q_d,$$
the tangent vectors  $\mathcal{V} \in T_{(q,p)}( T^*M)$ are identified with $(Q,P) \in T_qM \times T_qM^*$, with 
$$Q= Q_1 \partial q_1 + ... + Q_d \partial q_d, \quad P = P_1 \dd q_1 + ... + P_d \dd q_d.$$ With this notation,
$$\pi_* (Q,P) = Q,$$
$$\alpha_{(q,p)}(Q,P) = p(Q),$$
$$\omega_{(q,p)} ((Q,P),(Q',P')) = \langle P', Q\rangle - \langle P, Q' \rangle,$$
where $\langle ., . \rangle$ denotes the duality bracket between $T_qM$ and $T_qM^*$.

\begin{lemma}\label{j*omega=B}
$\Sigma$ is a symplectic submanifold of $(T^*M, \omega)$, and 
$$j^* \omega = B.$$
In particular, at each point $j(q) \in \Sigma$, 
\begin{align}\label{decompspace}
T_{j(q)}(T^* M) = T_{j(q)} \Sigma \oplus T_{j(q)} \Sigma^{\perp},
\end{align}
where $\perp$ denotes the symplectic orthogonal for $\omega$.
\end{lemma}

\begin{proof}
To say that $\Sigma$ is a symplectic submanifold of $T^*M$ means that the restriction of $\omega$ to $\Sigma$ is non-degenerate. Written with the embedding $j$, this restriction is $j^* \omega$. Actually, using the definition (\ref{defLiouville}) of $\alpha$ with $p = A_q$ and $\mathcal{V} = d_q j (Q)$, we get
$$\forall Q \in T_qM, \quad  (j^*\alpha)_{q}(Q) = A_q( \pi_* d_qj(Q)) = A_q(Q).$$
Hence $$j^* \alpha = A, \quad \text{so} \quad j^*(\dd \alpha) = \dd A = B.$$
\end{proof}

Since any $j(q)$ is a critical point of $H$, the Hessian of $H$ at $j(q)$ is well defined and independant of any choice of coordinates. We now compute this Hessian according to the decomposition (\ref{decompspace}):
\begin{lemma}\label{Hessian} The Hessian $T^2_{j(q)}H$, as a bilinear form on $T_{j(q)}(T^*M)$, can be written:
$$ T^2_{j(q)} H (\mathcal{V},\mathcal{V}) = 0 \quad \text{if} \quad \mathcal{V} \in T_{j(q)}\Sigma,$$
$$ T^2_{j(q)} H (\mathcal{V},\mathcal{V}) = 2 \vert \B(q) \pi_* \mathcal{V} \vert^2_{g_q} \quad \text{if} \quad \mathcal{V} \in T_{j(q)}\Sigma^{\perp}.$$
\end{lemma}

\begin{proof}
Using local coordinates on $M$, we will denote every $\mathcal{V} \in T_{(q,p)} (T^*M)$, as $(Q,P) \in T_qM \times T_qM^*$. In these coordinates, with the notations introduced in section \ref{sectioncoordinates}, 
$$\Sigma \equiv \lbrace (q, \A(q)), \quad q \in \R^d \rbrace$$
so that
\begin{align}\label{TSigmaCoordonnees}
T_{j(q)} \Sigma = \lbrace (Q,P)\in T_qM \times T_qM^*, P = T_qA \cdot Q  \rbrace.
\end{align}
We can also describe $T_{j(q)}\Sigma^{\perp}$ using these coordinates. Indeed,
\begin{align*}
(Q,P) \in T_{j(q)} \Sigma^{\perp} &\Leftrightarrow \forall Q_0 \in T_qM, \quad \omega( (Q,P), (Q_0,  T_qA \cdot Q_0) ) = 0 \\
& \Leftrightarrow \forall Q_0 \in T_qM, \quad \langle P, Q_0 \rangle = \langle T_q A \cdot Q_0,Q \rangle \\
& \Leftrightarrow P = \ ^tT_q A \cdot Q.
\end{align*}
Hence
\begin{align}\label{TSigmaPerpCoordonnees}
T_{j(q)}\Sigma^{\perp} = \lbrace (Q,P) , P = \ ^tT_qA \cdot Q \rbrace.
\end{align}
From the expression (\ref{HamiltonienCoordonnees}) of $H$ in coordinates, we deduce that:
\begin{align*}
T_{(q,p)}H (Q,P) =& \ 2 \sum_{ij} g^{ij}(q) (p_i - A_i(q))(P_j - \nabla_q A_j \cdot Q) \\ &+ \sum_{ijk} \partial_k g^{ij}(q) Q_k (p_i-A_i(q))(p_j-A_j(q)),
\end{align*}
so that the Hessian of $H$ in coordinates is:
\begin{align*}
T_{j(q)}^2H((Q,P),(Q,P)) &= 2 \sum_{ij} g^{ij}(q) (P_i - \nabla_q A_i \cdot Q)(P_j - \nabla_q A_j \cdot Q)\\ &= 2 \vert P - T_q A\cdot Q \vert^2_{g^*_q}.
\end{align*}
It follows from (\ref{TSigmaCoordonnees}) that
\begin{align*}
\forall (Q,P) \in T_{j(q)} \Sigma, \quad T^2_{j(q)}H((Q,P),(Q,P)) = 0,
\end{align*}
and from (\ref{TSigmaPerpCoordonnees}) and (\ref{iQBCoordonees}) that
\begin{align*}
\forall (Q,P) \in T_{j(q)} \Sigma^{\perp}, \quad T^2_{j(q)} H((Q,P),(Q,P)) &= 2 \vert  ( \  ^t T_q A - T_q A)Q \vert^2_{g^*_q}\\
&= \vert \iota_Q B \vert^2_{g^*_q}.
\end{align*}
Let us rewrite this using $\B$. Note that:
\begin{align*}
\vert \iota_Q B \vert^2_{g^*_q} &= \sum_{ij} g^{ij}(q) \left( \sum_{k i} B_{k i} Q_k \right) \left( \sum_{\ell j} B_{\ell j} Q_{\ell} \right) = \sum_{k \ell} \left( \sum_{ij} g^{ij}B_{ki}B_{\ell j} \right) Q_k Q_{\ell},
\end{align*}
and keeping in mind that $(g^{ij})$ is the inverse matrix of $(g_{ij})$ together with the relation (\ref{relationBBmatrice}) between $B$ and $\B$, we have
\begin{align*}
\sum_{ij} g^{ij} B_{ki}B_{\ell j} = \sum_{ijk' \ell'} g^{ij} g_{k' i} g_{\ell' j} \B_{k' k} \B_{\ell' \ell} = \sum_{k' \ell'} g_{k' \ell'} \B_{k' k} \B_{\ell' \ell},
\end{align*}
and so
\begin{align*}
\vert \iota_Q B \vert^2_{g^*_q} = \sum_{k' \ell'} g_{k' \ell'} \left( \sum_{k} \B_{k' k} Q_k \right) \left( \sum_{\ell} \B_{\ell' \ell} Q_{\ell}\right) = \vert \B(q) Q \vert^2_{g_q}.
\end{align*}
\end{proof}

We endow $\Omega \times \R^d_z$ with the symplectic form:
$$\omega_0(q,z) = B \oplus \sum_{j=1}^{d/2} \dd \xi_j \wedge \dd x_j,$$ with the notation $z = (x,\xi)$. $(\Sigma, B)$ is a $d$-dimensional symplectic submanifold of $(T^*M,\omega)$. The following Darboux-Weinstein lemma claims that this situation is modelled on the submanifold $\Sigma_0 = \Omega \times \lbrace 0 \rbrace$ of $(\Omega \times \R^d_z, \omega_0)$.

\begin{lemma} \label{Phizero}
There exists a local diffeomorphism
$$\Phi_0 : \Omega \times \R^d_z \rightarrow T^*M$$
such that $$\Phi_0^*\omega = \omega_0, \quad \text{and } \Phi_0( \Sigma_0 ) = \Sigma.$$
\end{lemma}

In order to keep track on the construction of $\Phi_0$, we will give the proof of this result.

\begin{proof}
Again, we use local coordinates on $M$ to denote every $\mathcal{V} \in T_{(q,p)}(T^*M)$ as $(Q,P) \in T_qM \times T^*_qM$. For $q \in \Omega$, using the vectors $u_j(q), v_j(q) \in T_qM$ defined in (\ref{defujvj}), we define the vectors
\begin{align*}
e_j(q) = \frac{1}{\sqrt{\beta_j(q)}} \left( u_j(q) , \ ^tT_qA \ u_j(q) \right), \quad  f_j(q) = \frac{1}{\sqrt{\beta_j(q)}} \left( v_j(q), \ ^tT_q A \ v_j(q) \right),
\end{align*}
which are in $T_{j(q)} \Sigma^{\perp}$ by (\ref{TSigmaPerpCoordonnees}). These vectors satisfy
\begin{align}\label{basesymplectique}
\omega_{j(q)}(e_i(q),f_j(q)) = \delta_{ij}, \quad \omega_{j(q)}(e_i(q),e_j(q)) = 0, \quad \omega_{j(q)}(f_i(q),f_j(q)) = 0.
\end{align}
Indeed, the first equality follows from
\begin{align*}
 \omega_{j(q)}(e_i,f_j) 
&= - \frac{1}{\sqrt{\beta_i \beta_j}} \langle ( \ ^t T_q A - T_q A ) u_j,v_j \rangle\\
&= - \frac{1}{\sqrt{\beta_i \beta_j}} B(u_i,v_j)\\
&= - \frac{1}{\sqrt{\beta_i \beta_j}} g_q(\B(q) u_i, v_j)\\
&= \frac{ \beta_i}{\sqrt{\beta_i \beta_j}} g_q(v_i,v_j)\\
&= \delta_{ij},
\end{align*}
and the two others from similar calculations.

Let us construct a $\tilde{\Phi}_0 : \Omega \times \R^d_z \rightarrow T^*M$ such that:
\begin{align}\label{DefPhiZeroTilde1}
\tilde{\Phi}_0(q,0) = j(q),\\
\label{DefPhiZeroTilde2}
\partial_z \tilde{\Phi}_0 (q,0) = L_q,
\end{align}
where $L_q : \R^d \rightarrow T_{j(q)} \Sigma^{\perp}$ is the linear map sending the canonical basis onto $$(e_1(q),f_1(q), ..., e_{d/2}(q),f_{d/2}(q)).$$
For this, we take local vector fields $\hat{e}_j(q,p), \hat{f}_j(q,p) \in T_{(q,p)} (T^*M)$ defined in a neighborhood of $\Sigma$, such that $$\hat{e}_j(j(q)) = e_j(q), \quad \hat{f}_j(j(q)) = f_j(q).$$
In other words, if we see $e_j$ and $f_j$ as vector fields on $\Sigma$ using $j(q)$, we extend them to a neighborhood of $\Sigma$. Then we consider the associated flows, defined on a neighborhood of $\Sigma$ by:
\begin{align*}
&\frac{\partial \phi_j^{x_j}}{\partial x_j}(q,p) = \hat{e}_j(\phi_j^{x_j}(q,p)), \quad x_j \in \R,\\
&\frac{\partial \psi_j^{\xi_j}}{\partial \xi_j}(q,p) = \hat{f}_j(\psi_j^{\xi_j}(q,p)), \quad \xi_j \in \R,\\
&\phi_j^0(q,p) = \psi_j^0(q,p)=(q,p).
\end{align*}
Then
$$\tilde{\Phi}_0(q,z) := \phi_1^{x_1} \circ \psi_1^{\xi_1} \circ ... \circ \phi_{d/2}^{x_{d/2}} \circ \psi_{d/2}^{\xi_{d/2}}(j(q))$$
satisfies (\ref{DefPhiZeroTilde1}) and (\ref{DefPhiZeroTilde2}). Hence,
if $q\in \Omega,$ the linear tangent map $$T_{(q,0)}\tilde{\Phi}_0 : T_qM \oplus \R^d \rightarrow T_{j(q)}\Sigma \oplus T_{j(q)}\Sigma^{\perp}$$ acts as:
$$\begin{pmatrix}
T_qj & 0 \\
0 & L_q
\end{pmatrix}.$$
In particular, $\tilde{\Phi}_0^*\omega = \omega_0$ on $\lbrace z=0 \rbrace$ by (\ref{basesymplectique}) and lemma \ref{j*omega=B}.
By Weinstein lemma \ref{RelativeDarboux} (Appendix), for $\varepsilon >0$ small enough there exists a diffeomorphism $S :  \Omega \times B_z(\varepsilon) \rightarrow  \Omega \times B_z(\varepsilon)$ such that $S(q,z) = (q,z) + \grandO (\vert z \vert^2)$ and $S^* \tilde{\Phi}_0^* \omega = \omega_0$. Then $\Phi_0 = \tilde{\Phi}_0 \circ S$ is the desired symplectomorphism.
\end{proof}

\subsection{Proof of Theorem \ref{symplectomorphism}}

Now we can prove the normal form for the classical Hamiltonian. Up to reducing $\Omega$, we can take symplectic coordinates $w=(y,\eta) \in \R^d$ to describe $\Omega$, thanks to the Darboux lemma:
$$\varphi : \Omega \rightarrow V \subset \R^d_w.$$
We get a new symplectomorphism
$$\Phi : V \times B_z(\varepsilon) \rightarrow U \subset T^*M,$$
defined by
$$\Phi(w,z) = \Phi_0(\varphi^{-1}(w),z).$$
It remains to compute a Taylor expansion of $H$ in these coordinates.
Using the Taylor Formula for $\hat{H}= H \circ \Phi$, we get:
\begin{align}\label{TaylorexpH00}
\hat{H}(w,z) = \hat{H}(w,0) + \partial_z \hat{H}_{\vert z=0} (z) + \frac{1}{2} \partial_z^2 \hat{H}_{\vert z=0} (z,z) + \grandO(\vert z \vert^3).
\end{align}
By the chain rule, we have (with $q = \varphi^{-1}(w)$):
$$\partial_z \hat{H}_{\vert z=0} (z) = T_{j(q)} H ( \partial_z \Phi_{\vert z=0} (z) ) = 0,$$
because $T_{j(q)}H = 0$, and
$$\partial_z^2 \hat{H}_{\vert z=0}(z,z) = T_{j(q)}^{ 2} H ( \partial_z \Phi_{\vert z=0} (z) , \partial_z \Phi_{\vert z=0} (z) ).$$
But $\partial_z \Phi_{\vert z=0} $ sends the canonical basis onto $(e_1(q), f_1(q), ... \ , e_{d/2}(q),  f_{d/2}(q) )$, so we get from Lemma \ref{Hessian}:
$$ \frac{1}{2}\partial_z^2 \hat{H}_{\vert z=0}(z,z) = \sum_{j=1}^{d/2} \beta_j(q) \vert z_j \vert^2. $$
Hence (\ref{TaylorexpH00}) gives:
$$\hat{H}(w,z) = H \circ \Phi (w,z) = \sum_{j=1}^{d/2} \hat{\beta}_j(w) \vert z_j \vert^2 + \grandO (\vert z \vert^3).$$

\section{\textbf{The Formal Birkhoff Normal Form}}

\subsection{The Hamiltonian $\hat{H}$}

In the new coordinates given by Theorem \ref{symplectomorphism}, we have a Hamiltonian $\hat{H}(w,z)$ of the form: 
$$\hat{H}(w,z) = H^0(w,z) + \grandO( \vert z \vert^3), \quad \text{ where } H^0(w,z) = \sum_{j=1}^{d/2} \hat{\beta}_j(w) \vert z_j \vert^2.$$

$H^0$ is defined for $w \in V$, but we extend the functions $\hat{\beta}_j$ to $\R_w^{d}$ such that:
\begin{align}\label{extentionH0}
\sum_{j=1}^{d/2} \hat{\beta}_j(w) \geq \tilde{b}_1 \quad \text{for } w \in V^c.
\end{align}
This is just technical, since we will prove microlocalization results on $V$ in section 6. Then we can construct a Birkhoff normal form, in the spirit of \cite{Sjo92} and \cite{article}, with $w$ as a parameter. 

\subsection{The space of formal series} \label{SectionFormalSeries}

We will work in the space of formal series $$\mathcal{E}= \mathcal{C}^{\infty}(\R^d_w) [[x,\xi,\hbar]].$$ We endow $\mathcal{E}$ with the Moyal product $\star$, compatible with the Weyl quantization (with respect to all the variables $z$ and $w$). Given a pseudodifferential operator $\mathcal{A}= \Op (a)$ we will denote $\symbolseries (\mathcal{A})$ or $[a]$ the formal Taylor series of $a$ at zero, in the variables $x$, $\xi$, $\hbar$. With this notation, the compatibility of $\star$ with the Weyl quantization means $$\symbolseries ( \mathcal{A}\mathcal{B}) =\symbolseries ( \mathcal{A})\star \symbolseries ( \mathcal{B}).$$ The reader can find the main results on $\hbar$-pseudodifferential operators in \cite{Martinez} or \cite{Zworski}.\\

We define the degree of $x^{\alpha} \xi^{\gamma} \hbar^{\ell}$ to be $\vert \alpha \vert + \vert \gamma \vert + 2\ell$. Hence, we can define the degree and valuation of a series $\kappa$, which depends on the point $w \in \R^d$. We denote $\grandO_N$ the space of formal series with valuation at least $N$ on $V$, and $\mathcal{D}_N$  the space spanned by monomials of degree $N$ on $V$ ($V \subset \R^d_w$ is given by Theorem \ref{symplectomorphism}). We denote $z_j$ the formal series $x_j + i \xi_j$. Thus every $\kappa \in \mathcal{E}$ can by written
$$ \kappa = \sum_{\alpha \gamma \ell} c_{\alpha \gamma \ell}(w) z^{\alpha} \bar{z}^{\gamma} \hbar^{\ell},$$
with the notation $$z^{\alpha} = z_1^{\alpha_1} ... z_{d/2}^{\alpha_{d/2}}.$$
For $\kappa_1$, $\kappa_2 \in \mathcal{E}$, we denote $\ad_{\kappa_1} \kappa_2 = [ \kappa_1, \kappa_2 ] = \kappa_1 \star \kappa_2 - \kappa_2 \star \kappa_1$. It is well known that $[\kappa_1,\kappa_2]$ is of order $\hbar$, so for $N_1+N_2 \geq 2$, we have 
\begin{align}\label{OrdreEtCommutateurs}
\frac{1}{\hbar} [\grandO_{N_1}, \grandO_{N_2}] \subset \grandO_{N_1+N_2-2}.
\end{align}
Explicitly, we have
\begin{align}\label{CommutatorOfPseudo}
[\kappa_1, \kappa_2 ](z,w,\hbar) = 2 \sinh \left( \frac{\hbar}{2i} \square \right) ( f(z',w',\hbar)g(z'',w'',\hbar)) \vert_{z'=z''=z, w'=w''=w},
\end{align}
where $[f]=\kappa_1$, $[g]= \kappa_2$, and $$\square = \sum_{j=1}^{d/2} \left( \partial_{\xi_j'} \partial_{x_j''} - \partial_{x_j'} \partial_{\xi_j''} + \partial_{\eta_j'} \partial_{y_j''} - \partial_{y_j'} \partial_{\eta_j''} \right).$$
From formula (\ref{CommutatorOfPseudo}), a simple computation yields to
\begin{align}\label{AdZj2=}
\frac{i}{\hbar} \ad_{\vert z_j \vert^2} ( z^{\alpha} \bar{z}^{\beta} \hbar^{\ell} ) = \lbrace \vert z_j \vert^2, z^{\alpha} \bar{z}^{\gamma} \hbar^{\ell} \rbrace = (\alpha_j - \gamma_j) z^{\alpha} \bar{z}^{\gamma} \hbar^{\ell}.
\end{align}

\subsection{The formal normal form}

In order to prove Theorem \ref{ThmFormeNormale}, we look for a pseudodifferential operator $\mathcal{Q}_{\hbar}$ such that
\begin{align}\label{30041900}
e^{\frac{i}{\hbar} \mathcal{Q}_{\hbar}} \Op \hat{H} e^{-\frac{i}{\hbar} \mathcal{Q}_{\hbar}}
\end{align}
 commutes with the harmonic oscillators $\Ih^{(j)}, (1 \leq j \leq d/2)$ introduced in (\ref{DefOscillateurs}). At the formal level, expression (\ref{30041900}) becomes 
\begin{align} \label{30041901}
 e^{\frac{i}{\hbar} \ad_{\tau}} ( H^0 + \gamma ),
\end{align} 
 where $H_0 + \gamma$ is the Taylor expansion of $\hat{H}$, and $\tau = \symbolseries (\mathcal{Q}_{\hbar}).$ Moreover, $$\symbolseries ( \Ih^{(j)} ) = \vert z_j \vert^2,$$ so we want (\ref{30041901}) to be equal to $H^0 + \kappa$, where $[\kappa, \vert z_j \vert^2] = 0$, which is equivalent to say that $\kappa$ is a series in $(\vert z_1 \vert^2, ..., \vert z_{d/2} \vert^2, \hbar)$. This is possible modulo $\grandO_r$, as stated in the following theorem. We recall that $r$ is the non-resonance order, defined in (\ref{defr}), and that we assumed $r \geq 3$.

\begin{theorem} \label{algtheorem}
If $\gamma \in \grandO_3$, there exist $\tau, \kappa, \rho \in \grandO_3$ such that:\\

$\bullet \quad e^{\frac{i}{\hbar} \ad_{\tau}} ( H^0 + \gamma )= H^0 + \kappa + \rho,$\\

$\bullet \quad [\kappa, \vert z_j \vert^2] = 0 \quad \text{for } 1 \leq j \leq d/2,$\\

$\bullet \quad \rho \in \grandO_r.$\\
\end{theorem}

\begin{proof}
Let $3 \leq N \leq r - 1$. Assume that we have, for a $\tau_N \in \grandO_3$:
$$e^{\frac{i}{\hbar} \ad_{\tau_N}} (H^0+\gamma) = H^0 + K_3 + ... + K_{N-1} + R_{N} + \grandO_{N+1},$$ where $K_i \in \mathcal{D}_i$ commutes with $\vert z_j \vert^2$ ($1 \leq j \leq d/2$) and where $R_{N} \in \mathcal{D}_{N}$. Using (\ref{OrdreEtCommutateurs}), we have for any $\tau' \in \mathcal{D}_{N}$:
\begin{align*}
e^{\frac{i}{\hbar} \ad_{\tau_N + \tau'}} (H^0+\gamma) &= e^{\frac{i}{\hbar} \ad_{\tau'}} \left( H^0 + K_3 + ... + K_{N-1} + R_{N} + \grandO_{N+1} \right) \\
&= H^0 + K_3 + ... + K_{N-1} + R_{N} + \frac{i}{\hbar} \ad_{\tau'} H^0 + \grandO_{N+1}.
\end{align*}
Thus, we look for $\tau'$ and $K_N \in \mathcal{D}_N$ such that:
\begin{align}\label{eqFormelRecurrence1}
R_{N} = K_{N} + \frac{i}{\hbar} \ad_{H^0} \tau' \quad \text{modulo } \grandO_{N+1}.
\end{align}
To solve this equation, we need to study $\ad_{H^0}$. Since $H^0 = \sum_j \hat{\beta}_j(w) \vert z_j \vert^2$,
$$\frac{i}{\hbar}\ad_{H^0} \tau' = \sum_{j=1}^{d/2} \left(  \bbeta_j(w) \frac{i}{\hbar} \ad_{\vert z_j \vert^2} (\tau') + \frac{i}{\hbar} \ad_{\bbeta_j} (\tau' ) \vert z_j  \vert^2 \right).$$
Since $\bbeta_j$ only depends on $w$,
$$\frac{i}{\hbar} \ad_{\bbeta_j} (\tau') \in \grandO_{N-1},$$
(see formula (\ref{CommutatorOfPseudo})). Hence
$$\frac{i}{\hbar}\ad_{H^0} \tau' = \sum_{j=1}^{d/2}  \bbeta_j(w) \frac{i}{\hbar} \ad_{\vert z_j \vert^2} (\tau') + \grandO_{N+1}. $$
Thus equation (\ref{eqFormelRecurrence1}) can be rewritten
\begin{align}\label{eqFormelReccurrence2}
R_N = K_N + T( \tau') + \grandO_{N+1},
\end{align}
with the notation
$$T  = \sum_{j=1}^{d/2} \bbeta_j(w) \frac{i}{\hbar} \ad_{\vert z_j \vert^2}.$$
From formula (\ref{AdZj2=}) we see that $T$ acts on monomials as 
\begin{align}\label{TActsOnMonomials}
T ( c(w) z^{\alpha} \bar{z}^{\gamma} ) = \langle \alpha - \gamma , \bbeta(w) \rangle c(w) z^{\alpha} \bar{z}^{\gamma}.
\end{align}
Thus, if we write $$R_{N} = \sum_{\vert \alpha \vert + \vert \gamma \vert + 2 \ell = N} r_{\alpha \gamma \ell}(w) z^{\alpha} \bar{z}^{\gamma} \hbar^{\ell},$$ we choose $$K_N = \sum_{\alpha = \gamma} r_{\alpha \gamma \ell} \vert z \vert^{2 \alpha} \hbar^{\ell},$$
which commutes with $\vert z_j \vert^2$ ( $1 \leq j \leq d/2$ ). The rest $R_N - K_N$ is a sum of monomials of the form $r_{\alpha \gamma \ell} z^{\alpha} \bar{z}^{\gamma} \hbar^{\ell}$ with $\alpha \neq \gamma$. As soon as $0 < \vert \alpha - \gamma \vert < r$, we have $\langle \alpha - \gamma , \hat{\beta}(w) \rangle \neq 0$ (by (\ref{fffffrrr}) because $r$ is lower than the resonance order (\ref{defrzero})), so we can define the smooth coefficient $$c_{\alpha \gamma \ell}(w) = \frac{r_{\alpha \gamma \ell }(w)}{\langle \alpha - \gamma , \hat{\beta}(w) \rangle}.$$
Thus (\ref{TActsOnMonomials}) yields to $$T(c_{\alpha \gamma \ell} z^{\alpha} \bar{z}^{\gamma} \hbar^{\ell}) = r_{\alpha \gamma \ell}(w) z^{\alpha} \bar{z}^{\gamma} \hbar^{\ell},$$ so $R_N - K_N$ is in the range of $T$ modulo $\grandO_{N+1}$ because $N \leq r-1$. Hence we solved equation (\ref{eqFormelReccurrence2}), and thus we can iterate until $N=r-1$. The series $\rho$ is the $\grandO_r$ that remains:
$$e^{i\hbar^{-1} \ad_{\tau_N}} (H^0+\gamma) = H^0 + K_3 + ... + K_{r-1} + \rho.$$
\end{proof}

\section{\textbf{The Semiclassical Birkhoff Normal Form}}
\label{SectionFormeNormale}

The next step is to quantize Theorems \ref{symplectomorphism} and \ref{algtheorem}.
 
\subsection{Quantization of Theorem \ref{symplectomorphism}}

Theorem \ref{symplectomorphism} gives a symplectomorphism $\Phi$ reducing $H$ to $\hat{H} = H \circ \Phi$. 
We can quantize this result in the following way. The Egorov Theorem (Thm 5.5.9 in \cite{Martinez}) implies the existence of a Fourier integral operator $$V_{\hbar} :  \Ld(\R^d_{(x,y)}) \rightarrow \Ld(M),$$ associated to the symplectomorphism $\Phi$, and a pseudo-differential operator $\Lhc$ with principal symbol $\hat{H}$ on $V \times B_z(\varepsilon)$ and subprincipal symbol $0$, such that:
\begin{align}\label{quantifsymplecto}
V_{\hbar}^* \Lh V_{\hbar} = \Lhc,
\end{align}
\begin{align}\label{microinv1}
V_{\hbar}^* V_{\hbar} = I \quad \text{microlocally on } V \times B_z(\varepsilon),
\end{align}
and
\begin{align}\label{microinv2}
V_{\hbar} V_{\hbar}^* = I \quad \text{microlocally on } U.
\end{align}

\subsection{Proof of Theorem \ref{ThmFormeNormale}}

By (\ref{quantifsymplecto}), we are reduced to the pseudodifferential operator $\Lhc$, which has a total symbol of the form
\begin{align}
\sigma_{\hbar} = \hat{H} + \hbar^2 \tilde{r}_{\hbar} \quad \text{on } V \times B_z(\varepsilon).
\end{align}
In particular, $\symbolseries \left( \Lhc \right) = H_0 + \gamma$ for some $\gamma \in \grandO_3$, with the notation of section \ref{SectionFormalSeries}.
We want to construct a normal form using a bounded pseudodifferential operator $\mathcal{Q}_{\hbar}$:
\begin{align}\label{30041902}
e^{\frac{i}{\hbar} \mathcal{Q}_{\hbar}} \Lhc e^{-\frac{i}{\hbar} \mathcal{Q}_{\hbar}} = \BNh + R_{\hbar}.
\end{align}
In Theorem \ref{algtheorem}, applied to $\gamma$, we have constructed formal series $\tau$, $\kappa$, and $\rho$ such that
$$ e^{\frac{i}{\hbar} \ad_{\tau}} ( H^0 + \gamma )= H^0 + \kappa + \rho.$$
The idea is to choose pseudodifferential operators $\mathcal{Q}_{\hbar}$ and $\BNh$ such that $\symbolseries \left(  \mathcal{Q}_{\hbar} \right) = \tau$ and $\symbolseries \left( \BNh \right) = \kappa$, and to check that they satisfy (\ref{30041902}). Following this idea, we prove the following Theorem.

\begin{theorem}\label{BNh}
For $\hbar \in (0,\hbar_0]$ small enough, there exist a unitary operator $$\Uh : \Ld (\R^d ) \rightarrow \Ld (\R^d),$$ a smooth function $f^{\star}(w,I_1,...,I_{d/2},\hbar)$, and a pseudodifferential operator $\Rh$ such that:
\begin{align*}
&(i)\quad \Uh^* \Lhc \Uh = \Lh^0 + \Op f^{\star}(w, \Ih^{(1)}, ..., \Ih^{(d/2)}, \hbar) + \Rh,\\
&(ii)\quad f^{\star} \text{ has an arbitrarily small compact } (I_1,...,I_{d/2},\hbar)\text{-support (containing 0),}\\
&(iii)\quad \symbolseries (\Rh) \in \grandO_r \quad \text{and} \quad \symbolseries (\Uh \Rh \Uh^*) \in \grandO_r.
\end{align*}
with $\Ih^{(j)} = \Op( \vert z_j \vert^2 )$ and $\Lh^0 = \Op ( H^0)$. We call 
\begin{align}\label{FormuleFormeNormale}
\BNh = \Lh^0 + \Op f^{\star}(w, \Ih^{(1)}, ..., \Ih^{(d/2)}, \hbar)
\end{align}
the normal form, and $\Rh$ the remainder.
\end{theorem}

\begin{proof}
The pseudodifferential operator $\Lhc$ defined by (\ref{quantifsymplecto}) has a symbol of the form
$$\sigma_{\hbar} = \hat{H} + \hbar^2 \tilde{r}_{\hbar} \quad \text{on } V \times B_z(\varepsilon),$$
so $\sigma_{\hbar} = H^0 + r_{\hbar}$ with $\gamma := [r_{\hbar} ] \in \grandO_3.$
We apply Theorem \ref{algtheorem} with this $\gamma \in \grandO_3$. The formal series $\kappa \in \grandO_3$ that we get commutes with $\vert z_j \vert^2$ ($1 \leq j \leq d/2$), so by formula (\ref{AdZj2=}) we can write it 
$$\kappa = \sum_{k \geq 2} \sum_{l+\vert m \vert=k} c_{l,m}(w) \vert z_1 \vert^{2m_1} ... \vert z_{d/2} \vert^{2m_{d/2}} \hbar^l,$$ 
and we can change the coefficients to get 
$$\kappa = \sum_{k \geq 2} \sum_{l+\vert m \vert=k} c^{\star}_{l,m}(w) (\vert \zz_1 \vert^2)^{\star m_1} ... (\vert z_{d/2} \vert^2)^{\star m_{d/2}} \hbar^l.$$
We define functions: $$f(w,I_1,...,I_{d/2},\hbar) \text{ with Taylor series } \sum_{k \geq 2} \sum_{l+\vert m \vert=k} c_{l,m}(w) I_1^{m_1}...I_{d/2}^{m_{d/2}} \hbar^l,$$ $$f^{\star}(w, I_1,..., I_{d/2},\hbar)  \text{ with Taylor series } \sum_{k \geq 2} \sum_{l+\vert m \vert=k} c^{\star}_{l,m}(w) I_1^{m_1}...I_{d/2}^{m_{d/2}} \hbar^l,$$ 
and arbitrarily small compact support in $(I_1,...,I_{d/2},\hbar)$ (containing $0$).

Let $c(w,z,\hbar)$ be a smooth function with compact support with Taylor series $\tau$, given by Theorem \ref{algtheorem}. Then by the Taylor formula, we have:
\begin{align*}
 &e^{\frac{i}{\hbar}\Op (c)} \Op (H^0+ r_{\hbar}) e^{-\frac{i}{\hbar} \Op (c) } = \sum_{n=0}^{r-1} \frac{1}{n!} \ad_{i\hbar^{-1}\Op(c)}^n \Op(H^0+ r_{\hbar})\\ &+ \int_0^1 \frac{1}{(r-1)!}(1-t)^{r-1} e^{it\hbar^{-1}\Op(c)} \ad_{i\hbar^{-1}\Op(c)}^{r} \Op(H^0+ r_{\hbar}) e^{-it\hbar^{-1}\Op(c)} \dd t.
\end{align*} 
By the Egorov Theorem and the fact that $ \ad_{i\hbar^{-1}\Op(c)}^{r} : \mathcal{E} \rightarrow \grandO_{r}$ (see (\ref{OrdreEtCommutateurs})), the integral remainder has a symbol with Taylor series in $\grandO_{r}$. Moreover,
\begin{align*}
\symbolseries \left( \sum_{n=0}^{r-1} \frac{1}{n!} \ad_{i\hbar^{-1}\Op(c)}^n \Op(H^0+ r_{\hbar}) \right) 
&= \sum_{n=0}^{r-1} \frac{1}{n!} \ad_{i\hbar^{-1} \tau}^n ( H^0 + \gamma )\\
&= e^{\frac{i}{\hbar} \ad_{\tau}}(H^{0} + \gamma ) + \grandO_{r} \\ &= H^{0} + \kappa + \grandO_{r}.
\end{align*}
Thus, by the definition of $f$, there exists $s(w,z,\hbar)$ such that $[s] \in \grandO_{r}$ and:
\begin{align*}
e^{\frac{i}{\hbar}\Op (c)} \Op (H^0+ r_{\hbar}) e^{-\frac{i}{\hbar} \Op (c) } = \Op(H^{0}) + \Op(f(w,\vert z_1 \vert^2, ..., \vert z_{d/2} \vert^2,\hbar)) + \Op(s).
\end{align*}
Using the compatibility of the quantization with the Moyal product, we have $$\symbolseries ( f^{\star}(w,\Ih^{(1)},...,\Ih^{(d/2)},\hbar) ) = [f(w,\vert z_1 \vert^2, ..., \vert z_{d/2} \vert^2,\hbar)],$$ so we get:
\begin{align*}
e^{\frac{i}{\hbar}\Op (c)} \Op (H^0+ r_{\hbar}) e^{-\frac{i}{\hbar} \Op (c) } = \Op(H^{0}) + \Op ( f^{\star}(w,\Ih^{(1)},...,\Ih^{(d/2)},\hbar)) + \Op(\tilde{s}),
\end{align*}
for a new symbol $\tilde{s}(w,z,\hbar)$ with $[\tilde{s}] \in \grandO_r$. Hence we get
$$\Uh^* \Lhc \Uh = \Op(H^{0}) + \Op ( f^{\star}(w,\Ih^{(1)},...,\Ih^{(d/2)},\hbar)) + \Op(\tilde{s}),$$
with $\Uh = e^{-\frac{i}{\hbar}\Op (c)}$. To prove $(iii)$ with $\Rh= \Op( \tilde{s}),$ note that
$$\symbolseries(\Rh) = [\tilde{s}] \in \grandO_r$$
and $$\symbolseries( \Uh \Rh \Uh^*) = e^{\frac{i}{\hbar} \ad_{\tau}}([\tilde{s}]) \in \grandO_r.$$
\end{proof}

Theorem \ref{ThmFormeNormale} follows with the new operator $\Uht = V_{\hbar} \Uh$ given by (\ref{quantifsymplecto}) and Theorem \ref{BNh}. Point $(ii)$ of Theorem \ref{ThmFormeNormale} is remaining. We prove it here, using that the function $f^{\star}$ can be chosen with arbitrarily small compact support.

\begin{proposition}\label{perturbationH0}
For any $\zeta \in (0,1)$, up to reducing the support of $f^{\star}$, the normal form $\BNh$ of Theorem \ref{BNh} satisfies for $\hbar \in (0,\hbar_0]$ small enough:
$$(1-\zeta) \langle \Lh^0 \psi, \psi \rangle \leq \langle \BNh \psi , \psi \rangle \leq (1+\zeta) \langle \Lh^0 \psi, \psi \rangle, \quad \forall \psi \in \mathcal{S}(\R^d).$$
\end{proposition}

\begin{proof}
For a given $K >0$, we can take a cutoff function $\chi$ supported in $\lbrace \lambda \in \R^{d/2} : \Vert \lambda \Vert \leq K \rbrace$, and change $f^{\star}$ into $\chi f^{\star}$. Thus, for $\lambda_j \in \spectre (\Ih^{(j)})$,
\begin{align*}
\vert \chi f^{\star}(w,\lambda_1,..., \lambda_{d/2},\hbar) \vert & \leq C K \Vert \lambda \Vert\\
& \leq CK \sum_j \frac{1}{\min \bbeta_j} \bbeta_j(w) \lambda_j  \\
& \leq \tilde{C} K \sum_{j} \bbeta_j(w) \lambda_j.
\end{align*}
Hence, using functional calculus and the G$\overset{\circ}{\text{a}}$rding inequality, we deduce that
\begin{align*}
\vert \langle \Op f^*(w,\Ih^{(1)},..., \Ih^{(d/2)},\hbar) \psi , \psi \rangle \vert &\leq \tilde{C} K \langle \Lh^0 \psi, \psi \rangle + c\hbar \Vert \psi \Vert^2\\
&\leq \zeta \langle \Lh^0 \psi,\psi \rangle,
\end{align*}
for $K$ and $\hbar$ small enough.
\end{proof} 

\section{\textbf{Spectral reduction of} $\BNh$}

In this section, we prove an expansion of the first eigenvalues of $\BNh$ in powers of $\hbar^{1/2}$. In order to prove Theorem \ref{expansionOftheEigenvalues}, it will only remain to compare the spectra of $\BNh$ and $\Lh$. This will be done in the next sections.\\

Let $1 \leq j \leq d/2$. For $n_j \geq 0$, we denote $h_{n_j} : \R \rightarrow \R$ the $n_j$-th Hermite function of the variable $x_j$. In particular, for every $1 \leq j \leq d/2$ we have:
\begin{align}\label{eigenfunctionsOfIh}
\Ih^{(j)}h_{n_j}(x_j) = \hbar (2n_j+1) h_{n_j}(x_j).
\end{align}
Moreover, $(h_{n_j})_{n_j \geq 0}$ is a Hilbertian basis of $\Ld(\R_{x_j})$:
$$\Ld(\R_{x_j}) = \bigoplus_{n_j \geq 0} \langle h_{n_j} \rangle.$$
On $\R^{d/2}_x$, we define the functions $\mathbf{h}_{n}$ for any $n=(n_1, ..., n_{d/2}) \in \N^{d/2}$ by
$$\mathbf{h}_n(x) = h_{n_1} \otimes ... \otimes h_{n_{d/2}} (x) = h_{n_1}(x_1) ... h_{n_{d/2}}(x_{d/2}).$$
We have the following space decomposition:
$$\Ld(\R^{d/2}_x) = \bigoplus_{n \in \N^{d/2}} \langle \mathbf{h}_n \rangle.$$
In particular, we have:
\begin{align}\label{decompositionOfL2space}
\Ld(\R^d_{x,y}) =  \bigoplus_{n \in \N^{d/2}} \left( \Ld (\R^{d/2}_y)  \otimes \langle \mathbf{h}_n \rangle \right).
\end{align}

Since $\BNh$ commutes with the harmonic oscillators $\Ih^{(j)}$ $(1 \leq j \leq d/2)$, it is reduced in the decomposition (\ref{decompositionOfL2space}). More precisely,

\begin{lemma}\label{reductionBNHdiagonale}
For $n=(n_1,...,n_{d/2}) \in \N^{d/2}$, there exists a classical pseudodifferential operator $\BNh^{(n)}$ acting on $\Ld (\R_y^{d/2})$ such that:
$$\BNh(u \otimes h_{n_1} \otimes ... \otimes h_{n_{d/2}} ) = \BNh^{(n)}(u) \otimes h_{n_1} \otimes ... \otimes h_{n_{d/2}}, \quad \forall u \in \mathcal{S}(\R^{d/2}_y).$$
Its symbol is:
$$F^{(n)}(w) = \hbar \sum_{j=1}^{d/2} \bbeta_j(w)(2n_j+1) + f^{\star}(w,\hbar(2n+1),\hbar),$$
and we have:
$$\spectre (\BNh) = \bigcup_n \spectre (\BNh^{(n)}).$$
Moreover, the multiplicity of $\lambda$ as eigenvalue of $\BNh$ is the sum over $n$ of the multiplicities of $\lambda$ as eigenvalue of $\BNh^{(n)}$.
\end{lemma}

This follows directly from (\ref{eigenfunctionsOfIh}) and (\ref{FormuleFormeNormale}). Moreover, we can prove the following more precise inclusions of the spectra.

\begin{lemma}
Let $b_1 \in (b_0, \tilde{b}_1)$. There exist $\hbar_0, n_{max}, c>0$ such that, for any $\hbar \in (0, \hbar_0)$:
\begin{align}\label{decompositionSpectreBNh}
\spectre ( \BNh ) \cap (- \infty, b_1 \hbar ] \subset \bigcup_{0 \leq \vert n \vert \leq n_{max}} \spectre ( \BNh^{(n)} ),
\end{align}
and for any $n \in \N^{d/2}$ with $1 \leq \vert n \vert \leq n_{max}$:
\begin{align}\label{inclusionSpectreBNhn}
\spectre (\BNh^{(n)}) \subset [\hbar (b_0 + c\vert n \vert),+\infty).
\end{align}
\end{lemma}

\begin{proof}
Remember that the functions $\hat{\beta}_j$ are bounded from below by a positive constant. Thus, the G$\overset{\circ}{\text{a}}$rding inequality implies that there are $\hbar_0,c>0$ such that, for every $\hbar \in (0,\hbar_0)$,
\begin{align}\label{GardingOnBetaj}
\langle \Op ( \hat{\beta}_j ) u , u \rangle \geq c \Vert u \Vert^{2}, \quad \forall u \in \Ld(\R^{d/2}_y).
\end{align}
For any $n \in \N^{d/2}$, we have:
\begin{align*}
\langle \BNh^{(n)} u , u \rangle &= \langle \BNh ( u \otimes \mathbf{h}_n ), u \otimes \mathbf{h}_n \rangle \\ &\geq (1-\zeta) \langle \Lh^{0} ( u \otimes \mathbf{h}_n), u \otimes \mathbf{h}_n \rangle \quad \text{by Proposition } \ref{perturbationH0} \\ &= (1- \zeta) \sum_{j=1}^{d/2} \hbar(2n_j +1 ) \langle \Op ( \hat{\beta}_j ) u , u \rangle
\end{align*}
because $\Lh^0 = \sum_j \Op( \bbeta_j) \Ih^{(j)}$. Thus using (\ref{GardingOnBetaj}) and the G$\overset{\circ}{\text{a}}$rding inequality,
\begin{align*}
\langle \BNh^{(n)} u , u \rangle 
&\geq \hbar (1-\zeta) (2c \vert n \vert \Vert u \Vert^2 + \langle \Op( \hat{b} ) u,u \rangle )\\
&\geq \hbar (1-\zeta)  (2c \vert n \vert +b_0 - \tilde{c}\hbar) \Vert u \Vert^{2}.
\end{align*}
This proves (\ref{inclusionSpectreBNhn}) for a new $c>0$. Moreover, if you take any eigenpair $(\lambda,\psi)$ of $\BNh$ with $\lambda \leq b_1 \hbar$, it is an eigenpair of some $\BNh^{(n)}$, with $\psi = u \otimes \mathbf{h}_n$, and:
\begin{align*}
\hbar (1-\zeta) (2 c \vert n \vert +b_0 - \tilde{c} \hbar) \Vert u \Vert^{2} \leq \langle \BNh ^{(n)} u , u \rangle = \langle \BNh \psi , \psi \rangle \leq b_1 \hbar \Vert \psi \Vert^{2}.
\end{align*}
Thus, there is a $n_{max}>0$ independent of $\hbar, \lambda, \psi$ such that
$$\vert n \vert \leq n_{max}.$$
We deduce (\ref{decompositionSpectreBNh}).
\end{proof}

Using the previous Lemma and the well-known expansion of the first eigenvalues of $\Op( \hat{b})$, we deduce an expansion of the first eigenvalues of $\BNh$.

\begin{theorem} \label{wellexpansionBNh} Let $\varepsilon >0$ and $N \geq 1$. There exist $\hbar_0 >0$ and $c_0 >0$ such that, for $\hbar \in (0,\hbar_0]$, the $N$ first eigenvalues of $\BNh$ : $(\lambda_j(\hbar))_{1 \leq j \leq N}$ admit an expansion in powers of $\hbar^{1/2}$ of the form:
$$\lambda_j(\hbar) = \hbar b_0 + \hbar^2(E_j+c_0) + \hbar^{5/2} c_{j,5} + \hbar^{ 3 } c_{j,6} + ... ,$$
where $\hbar E_j$ is the $j$-th eigenvalue of the $d/2$-dimensional harmonic oscillator associated to the Hessian of $\hat{b}$ at $0$, counted with multiplicity.
\end{theorem}

\begin{proof}
The smallest eigenvalues of $\BNh$ are those of $\BNh^{(0)}$, which has the symbol
\begin{align*}
\hbar \hat{b}(w) + f^{\star}(w,\hbar, ..., \hbar ) &= \hbar (\hat{b}(w) + \hbar c_0 + \grandO(\hbar^2)).\\
\end{align*}
The first eigenvalues of a semiclassical pseudodifferential operator with principal symbol $\hat{b}$ (which admits a unique and non-degenerate minimum) have an expansion of the form:
\begin{align}\label{ExpansionForBh}
 \mu_j(\hbar) = b_0 + \hbar E_{j} + \hbar^{3/2} \sum_{m \geq 0} a_{j,m} \hbar^{m/2},
 \end{align}
where $\hbar E_j$ is the $j$-th eigenvalue of the $d/2$-dimensional harmonic oscillator associated to the Hessian of $\hat{b}$ at the minimum. Let us recall the idea of the proof of this result. Since the minimum of $\hat{b}$ is non degenerate, we can write
$$\hat{b}(w) = b_0 + \frac{1}{2}\Hess_0 \hat{b} (w,w) + \grandO(\vert w \vert^3).$$
A linear symplectic change of coordinates changes $\Hess_0 \hat{b}$ into
$$\sum_{j=1}^{d/2} \nu_j ( y_j^2 + \eta_j^2),$$
for some positive numbers $(\nu_j )_{1 \leq j \leq d/2}$. In these coordinates the symbol becomes
$$\hat{b}(y,\eta) = b_0 + \sum_{j=1}^{d/2}\nu_j( y_j^2 + \eta_j^2) + \grandO( \vert w \vert^3) + \grandO(\hbar),$$
and Helffer-Sjöstrand proved in \cite{HS1} that the first eigenvalues of a pseudo-differential operator with such a symbol admits an expansion in powers of $\hbar^{1/2}$. Sjöstrand \cite{Sjo92} recovered this result using a Birkhoff normal form in the case where the coefficients $(\nu_j)_j$ are non-resonant. Charles and Vu Ngoc also tackled the resonant case in \cite{San}.
\end{proof}

\section{\textbf{Microlocalization results}}
\label{SectionMicrolocalisation}

In section \ref{SectionFormeNormale}, we have proved Theorem \ref{ThmFormeNormale}: We have constructed a normal form, which is only valid on a neighborhood $U$ of $\Sigma = H^{-1}(0)$ since the rest $R_{\hbar}$ can be large outside this neighborhood. Hence, we now prove that the eigenfunctions of $\Lh$ and $\BNh$ are microlocalized on a neighborhood of $\Sigma$.

\subsection{Microlocalization of the eigenfunctions of $\Lh$}

We recall that $$K = \lbrace b(q) \leq \tilde{b}_1 \rbrace \subset \Omega.$$ For $\varepsilon >0$, we denote
\begin{align}\label{DefKEpsilon}
K_{\varepsilon}= \lbrace q : \dd (q,K) \leq \varepsilon \rbrace.
\end{align}
For $\varepsilon >0$ small enough, $K_{\varepsilon} \subset \Omega$.

The following Theorem states the well-known Agmon estimates (see Agmon's paper \cite{Agmon}), which gives exponential decay of the eigenfunctions of the magnetic Laplacian $\Lh$ outside the minimum $q_0$ of the magnetic intensity $b$. In particular, these eigenfunctions are localized in $\Omega$.

\begin{theorem}[Agmon estimates]\label{Agmon} Let $\alpha \in (0,1/2)$ and $b_0 < b_1<\tilde{b}_1$. There exist $C, \hbar_0 >0$ such that for all $\hbar \in (0,\hbar_0]$ and for all eigenpair $(\lambda, \psi)$ of $\Lh$ with $\lambda \leq \hbar b_1 $, we have:
$$\int_M \vert e^{d(q,K)\hbar^{-\alpha}} \psi \vert ^2 \dd q \leq C \Vert \psi \Vert^2.$$ In particular, if $\chi_0 : M \rightarrow [0,1]$ is a smooth function being $1$ on $K_{\varepsilon}$,
$$ \psi = \chi_0 \psi + \grandO(\hbar^{\infty}) \quad \text{in } \Ld(M).$$
\end{theorem}

\begin{proof}
If $\Phi : M \rightarrow \R$ is a Lipschitz function such that $e^{\Phi} \psi$ belongs to the domain of $q_{\hbar}$, the Agmon formula (Theorem \ref{AgmonFormula} in Appendix),
$$q_{\hbar}(e^{\Phi} \psi) = \lambda \Vert e^{\Phi} \psi \Vert^2 + \hbar^2 \Vert \dd \Phi e^{\Phi} \psi \Vert^2,$$
together with the Assumption \ref{InequalityIntensity},
$$(1+\hbar^{1/4}C_0)q_{\hbar}(e^{\Phi}\psi) \geq \int \hbar (b(q) -  \hbar^{1/4} C_0) \vert e^{\Phi} \psi \vert^2 \dd q_g,$$ yields to:
$$\int \left[ \hbar (b(q) - \hbar^{1/4}C_0) - (1+\hbar ^{1/4} C_0) ( \lambda + \hbar^2 \vert \dd \Phi \vert^2 ) \right] \vert e^{\Phi} \psi \vert^2  \dd q_g \leq 0.$$
We split this integral into two parts:
\begin{align*}
\int_{K^c} & \left[ \hbar(b(q) - \hbar^{1/4}C_0) - (1+\hbar ^{1/4} C_0) ( \lambda + \hbar^2 \vert \dd \Phi \vert^2 ) \right]  \vert e^{\Phi} \psi \vert^2 \dd q_g \\ &\leq \int_{K} \left[ - \hbar (b(q)-\hbar^{1/4}C_0) + (1+\hbar ^{1/4} C_0) ( \lambda + \hbar^2 \vert \dd \Phi \vert^2 ) \right] \vert e^{\Phi} \psi \vert^2 \dd q_g.
\end{align*}
We choose $\Phi$:
$$\Phi_m(q) = \chi_m(d(q,K)) \hbar^{-\alpha} \quad \text{for } m >0,$$
where $\chi_m(t) = t$ for $t <m$, $\chi_m(t) =0$ for $t > 2m$, and $\chi_m'$ uniformly bounded with respect to $m$. Since $\Phi_m(q) =0$ on $K$ and $b(q)- \hbar^{1/4} C_0 \geq 0$, we have:
$$\int_{K^c} \left[ \hbar (b(q) - \hbar^{1/4}C_0) - (1+\hbar^{1/4}C_0)(\lambda + \hbar^2 \vert \dd \Phi_m \vert^2 ) \right] \vert e^{\Phi_m} \psi \vert^2 \dd q_g \leq C \hbar \Vert \psi \Vert^2.$$
Morever, $\lambda \leq b_1 \hbar$ and $\vert \dd \Phi_m \vert^2 \leq \tilde{C} \hbar^{-2 \alpha}$:
$$\int_{K^c} \left[ \hbar (b(q) - \hbar^{1/4}C_0) - (1+\hbar^{1/4}C_0)(b_1 \hbar + \tilde{C}\hbar^{2-2 \alpha}) \right] \vert e^{\Phi_m} \psi \vert^2 \dd q_g \leq C \hbar \Vert \psi \Vert^2.$$
Thus, up to changing the constant $C_0$:
$$\int_{K^c} \hbar (\tilde{b}_1 - b_1 - \hbar^{1/4}C_0 - \tilde{C} \hbar^{1-2\alpha}) \vert e^{\Phi_m}\psi \vert^2 \dd q \leq C \hbar \Vert \psi \Vert^2.$$ Since $\tilde{b}_1 > b_1$, we have $\tilde{b}_1 - b_1 - \hbar^{1/4}C_0 - \tilde{C} \hbar^{1-2\alpha} >0$ for $\hbar$ small enough. Hence  $$\int_{K^c} \vert e^{\Phi_m} \psi \vert^2 \dd q \leq C \Vert \psi \Vert^2,$$ and since $\Phi_m=0$ on $K$:
$$\int \vert e^{\Phi_m} \psi \vert^2 \dd q \leq (C+1) \Vert \psi \Vert^2.$$
By Fatou's lemma in the limit $m \rightarrow +\infty$,
$$\int \vert e^{d(q,K) \hbar^{-\alpha}} \psi \vert^2 \dd q \leq (C+1) \Vert \psi \Vert^2.$$ To prove the second result, notice that
\begin{align*}
\Vert \psi - \chi_0 \psi \Vert^2 &= \int_{\chi_0 \neq 1} \vert (1- \chi_0) \psi \vert^2 \dd q \leq \int_{\chi_0 \neq 1} \vert \psi \vert^2\dd q\\
& \leq \int_{K_{\varepsilon}^c} \vert \psi \vert^2 \dd q\\
& \leq e^{-2 \varepsilon \hbar^{-\alpha}} \int_{K_{\varepsilon}^c} \vert e^{d(q,K)\hbar^{-\alpha}} \psi \vert ^2 \dd q\\
& \leq C e^{-2 \varepsilon \hbar^{-\alpha}} \Vert \psi \Vert ^2 = \grandO (\hbar^{\infty}).
\end{align*}
\end{proof}

Now we prove the microlocalization of the eigenfunctions of $\Lh$ near $\Sigma$.

\begin{theorem} \label{microloc} Let $\varepsilon>0$, $\delta \in (0,\frac{1}{2})$, and $0<b_1<\tilde{b}_1$. Let $\chi_0 : M \rightarrow [0,1]$ be a smooth function being $1$ on $K_{\varepsilon}$. Let $\chi_1 : \R \rightarrow [0,1]$ be a smooth compactly supported cutoff function being $1$ near $0$. Then for any normalized eigenpair $(\lambda, \psi)$ of $\Lh$ such that $\lambda \leq \hbar b_1$ we have:
$$ \psi = \chi_1(\hbar^{-2\delta} \Lh) \chi_0(q) \psi + \grandO(\hbar^{\infty}) \quad \text{in } \Ld(M).$$
\end{theorem}

\begin{proof}
Using Theorem \ref{Agmon}, we have $\psi = \chi_0 \psi + \grandO(\hbar^{\infty})$ in $\Ld(M)$. Since $\chi_1(\hbar^{-2\delta} \Lh)$ is a bounded operator, we get:
$$\chi_1(\hbar^{-2\delta} \Lh) \psi = \chi_1(\hbar^{-2\delta} \Lh) \chi_0 \psi + \grandO(\hbar^{\infty}) \quad \text{in } \Ld(M).$$
In fact, $$\psi = \chi_1(\hbar^{-2\delta} \Lh) \psi.$$
Indeed, there exists a $C>0$ such that $$\chi_1(\hbar^{-2 \delta} . \ ) = 1 \quad \text{on } B(0,C\hbar^{2\delta}),$$
and for $\hbar \in (0,\hbar_0)$ small enough, $$\lambda \in B(0,b_1\hbar) \subset B(0,C \hbar^{2\delta}).$$ Thus,
$$\chi_1(\hbar^{-2\delta} \Lh) \psi = \chi_1(\hbar^{-2\delta} \lambda) \psi = \psi.$$
\end{proof}

\subsection{Microlocalization of the eigenfunctions of $\BNh$}

The next two theorems states the microlocalization of the eigenfunctions of the normal form. We recall that if $\varphi$ is defined by Theorem \ref{symplectomorphism}, we have:
$$\varphi(K) = \lbrace w \in V : \hat{b}(w) \leq \tilde{b}_1 \rbrace,$$
with $\hat{b}(w) = b \circ \varphi^{-1}(w).$ We also recall the definition (\ref{DefKEpsilon}) of $K_{\varepsilon}$. This first lemma gives a microlocalization result on the $w$ variable.

\begin{lemma}\label{microBNh1}
Let $\hbar \in (0, \hbar_0]$ and $b_1 \in (0, \tilde{b}_1)$. Let $\chi_0$ be a smooth cutoff function on $\R^{d}_w$ supported on $V$ such that $\chi_0 = 1$ on $\varphi (K_{\varepsilon})$. Then for any normalized eigenpair $(\lambda,\psi)$ of $\BNh$ such that $\lambda \leq \hbar b_1$, we have:
\begin{align*}
\psi = \Op( \chi_0 ) \psi + \grandO ( \hbar^{\infty} ) \quad \text{in } \Ld (\R^d_{x,y}).
\end{align*}
\end{lemma}

\begin{proof}
Let $\chi = 1- \chi_0$, which is supported in $\varphi (K_{\varepsilon})^c$. The eigenvalue equation yields to
\begin{align} \label{01051902}
\langle \BNh \Op ( \chi ) \psi, \Op( \chi) \psi \rangle \leq b_1 \hbar \Vert \Op (\chi) \psi \Vert^2 + \langle [ \BNh , \Op ( \chi ) ] \psi, \Op ( \chi ) \psi \rangle.
\end{align}
Using Lemma \ref{reductionBNHdiagonale}, we can write $\psi = u \otimes \mathbf{h}_n$ for some $n \in \N^{d/2}$, $u \in \Ld (\R^{d/2}_w)$, with $0 \leq \vert n \vert \leq n_{max}$. Then
\begin{align*}
[ \BNh , \Op (\chi ) ] \psi &= [ \BNh^{(n)} , \Op ( \chi ) ] u \otimes \mathbf{h}_n \\ 
&=  \hbar \left[ \sum_{j=1}^{d/2} (2n_j + 1) \Op ( \hat{\beta}_j ), \Op (\chi ) \right] \psi + \grandO (\hbar^2),
\end{align*}
because the principal symbol of $\BNh^{(n)}$ is $\sum_{j=1}^{d/2} \hbar (2n_j + 1) \hat{\beta}_j$. Since the symbol of the commutator is of order $\hbar$ and supported in $\supp \chi$, we have
\begin{align}\label{01051901}
\langle [ \BNh , \Op (\chi ) ] \psi , \Op (\chi ) \psi \rangle \leq C \hbar^2 \Vert \Op(\bar{\chi}) \psi \Vert^2,
\end{align}
where $\bar{\chi}$ is a small extension of $\chi$, with value $1$ on $\supp \chi$ and $0$ on a neighborhood of $\varphi( K_{\varepsilon} )$.
Moreover using Proposition \ref{perturbationH0}, 
\begin{align*}
\langle \BNh \Op ( \chi )\psi , \Op ( \chi ) \psi \rangle &\geq (1-\zeta) \langle \Lh^0 \Op (\chi) \psi, \Op ( \chi) \psi \rangle \\ & \geq (1- \zeta) \hbar \tilde{b}_1 \Vert \Op ( \chi ) \psi \Vert^2,
\end{align*}
where we used the G$\overset{\circ}{\text{a}}$rding inequality because, the symbol of $\Lh^0$ is greater than $\tilde{b}_1$ on $\supp \chi$. Together with (\ref{01051902}) and (\ref{01051901}), we get
\begin{align*}
 \hbar \left( (1-\zeta) \tilde{b}_1 - b_1 \right) \Vert \Op ( \chi ) \psi \Vert^2 \leq C \hbar^2 \Vert \Op ( \bar{ \chi } ) \psi \Vert^2.
\end{align*}
For $\eta$ small enough, $(1-\zeta) \tilde{b}_1 > b_1$. Hence, dividing by $\hbar$ and iterating with $\bar{\chi}$ instead of $\chi$, we get
$$\Vert \Op( \chi ) \psi \Vert^2 = \grandO ( \hbar^{\infty} ).$$
\end{proof}

Now we prove the microlocalization of the eigenfunctions of $\BNh$ on a neighborhood of $\varphi(\Sigma) = \lbrace (z,w) : z=0 \rbrace$.

\begin{theorem} \label{microloc2} Let $\hbar \in (0,\hbar_0]$, $b_1 \in (0, \tilde{b}_1)$, and $\delta \in (0,1/2)$. Let $\chi_0$ be a smooth cutoff function on $\R^{d/2}_w$ supported on $V$ such that $\chi_0=1$ on $\varphi(K_{\varepsilon})$ and $\chi_1$ a real cutoff function being $1$ near $0$. Then for any normalized eigenpair $(\lambda, \psi)$ of $\BNh$ such that $\lambda \leq \hbar b_1$, we have:
$$ \psi = \chi_1(\hbar^{-2\delta} \Ih^{(1)})...\chi_1 (\hbar^{-2\delta} \Ih^{(d/2)}) \Op (\chi_0(w)) \psi  + \grandO(\hbar^{\infty}) \quad \text{in } \Ld(\R^d).$$
\end{theorem}

\begin{proof}
According to Lemma \ref{microBNh1}, 
$$\psi = \Op (\chi_0) \psi + \grandO(\hbar^{\infty}).$$
Since $\chi_1^{d/2}(\hbar^{-2 \delta} \Ih) := \chi_1(\hbar^{-2\delta} \Ih^{(1)})...\chi_1 (\hbar^{-2\delta} \Ih^{(d/2)})$ is a bounded operator, we have
\begin{align*}
\chi_1^{d/2}(\hbar^{-2 \delta} \Ih) \psi = \chi_1^{d/2}(\hbar^{-2 \delta} \Ih) \Op(\chi_0) \psi + \grandO( \hbar^{\infty}).
\end{align*}
It remains to prove that $\psi = \chi_1^{d/2}(\hbar^{-2 \delta} \Ih) \psi$ for $\hbar$ small enough. Using Lemma \ref{reductionBNHdiagonale}, $\psi = u \otimes \mathbf{h}_n$ for some $u \in \Ld(\R^{d/2}_{y}) $, $n \in \N^{d/2}$ with $0 \leq \vert n \vert \leq n_{max}$, and so
$$\chi_1^{d/2}(\hbar^{-2 \delta} \Ih)\psi = \chi_1(\hbar^{1-2\delta} (2n_1 + 1) ) ... \chi_1(\hbar^{1-2\delta} (2n_{d/2} + 1) ) \psi.$$
But $\chi_1 = 1$ on a neighborhood of $0$, so there is $\hbar_0>0$ such that, for any $\hbar \in (0,\hbar_0]$ and any $0 \leq \vert n \vert \leq n_{max}$, $$\chi_1(\hbar^{1-2\delta} (2n_1 + 1) ) ... \chi_1(\hbar^{1-2\delta} (2n_{d/2} + 1) )=1.$$
Thus, $$\psi = \chi_1^{d/2}(\hbar^{-2 \delta} \Ih)\psi.$$
\end{proof}

\subsection{Rank of the spectral projections}

We want the microlocalization Theorems \ref{microloc} and \ref{microloc2} to be uniform with respect to $\lambda \in (-\infty, b_1 \hbar]$. That is why we need the rank of the spectral projections to be bounded by some finite power of $\hbar^{-1}$. If $\mathcal{A}$ is a bounded from below self-adjoint operator, and $\alpha \in \R$, we denote $N(\mathcal{A}, \alpha)$ the number of eigenvalues of $\mathcal{A}$ smaller than $\alpha$, counted with multiplicities. It is the rank of the spectral projection $\mathbf{1}_{]-\infty, \alpha ]} (\mathcal{A})$.\\

The proof of the following estimate is inspired by the proof of Lemma \ref{manifoldinequality} in Appendix, adapted from \cite{HeMo96}. The idea is to locally approximate the magnetic field to a constant.

\begin{lemma}\label{rank} Let $b_0 < b_1 < \tilde{b}_1$. There exists $C>0$ and $\hbar_0>0$ such that for all $\hbar \in (0,\hbar_0]$, we have:
$$N(\Lh, \hbar b_1) \leq C\hbar^{-d/2}.$$
\end{lemma}

\begin{proof}
Take $(\chi_m)_{m \geq 0}$ a smooth partition of unity, such that:
$$\sum_{m \geq 0} \chi_m^2(q) = 1 \quad \text{and} \quad \sum_{m \geq 0} \vert \dd \chi_m(q) \vert ^2 \leq C, \quad \forall q \in M,$$
with $\supp (\chi_m) \subset \mathcal{V}_m$ a local chart. Then, by Lemma \ref{CalculQhPartitionUnity} (in Appendix), for any $\psi \in D(q^{\hbar})$,
$$q_{\hbar} (\psi) = \sum_{m \geq 0} q_{\hbar}(\chi_m \psi) - \hbar^2 \sum_{m \geq 0} \Vert \psi \dd \chi_m  \Vert^2 \geq \sum_{m \geq 0} q_{\hbar}(\chi_m \psi) - C \hbar^2 \Vert \psi \Vert^2.$$
Since $$K = \lbrace b(q) \leq \tilde{b}_1 \rbrace$$ is compact, there is a $m_0>0$ such that, for $m > m_0$:
\begin{align}
q_{\hbar}(\chi_m \psi) &\geq \hbar \int \left( b(q) - \hbar^{1/4}C_0 \right) \vert \chi_m \psi \vert^{2} \dd q_g \\ 
\label{eq700519} q_{\hbar}(\chi_m \psi) &\geq \hbar (\tilde{b}_1-\hbar^{1/4}C_0) \Vert \chi_m \psi \Vert^2 \geq \hbar b_1 \Vert \chi_m \psi \Vert^2,
\end{align}
for $\hbar$ small enough.
For $0 \leq m \leq m_0$, we can work like in $\R^d$ using the charts, and we can find a new partition of unity on $\mathcal{V}_m$ such that
\begin{align}\label{qdrttttt}
\sum_{j = 0}^J \vert \chi_{m,j}^{\hbar} \vert ^2 =1, \quad \text{and} \quad \sum_{j = 0}^J \vert \dd \chi_{m,j}^{\hbar}(x) \vert^2 \leq C \hbar^{-2\alpha},
\end{align}
where $C>0$ does not depend on $m$, and with $$\supp (\chi_{m,j}^{\hbar}) \subset \mathcal{B}_{m,j} := \lbrace x : \vert x-z_{m,j} \vert \leq \hbar^{\alpha} \rbrace.$$
Thus we have for $0 \leq m \leq m_0$:
\begin{align}\label{eqstar0519}
q_{\hbar}(\chi_m \psi) \geq \sum_{j=0}^J q_{\hbar}(\chi_{m,j} \psi ) - C \hbar^{2-2\alpha} \Vert \chi_m \psi \Vert^2.
\end{align}
On each $\mathcal{B}_{m,j}$, we will approximate the magnetic field by a constant. Up to a gauge transformation, we can assume that the vector potential vanishes at $z_{m,j}$. In other words, we can find a smooth function $\varphi_{m,j}$ on $\mathcal{B}_{m,j}$ such that $$\tilde{\A}(z_{m,j})=0,$$
where $\tilde{\A} = \A + \nabla \varphi_{m,j}.$ The potential $\tilde{\A}$ defines the same magnetic field $\B$ as $\A$. Let us define
$$\A_{lin}(x) = \B (z_{m,j}). (x-z_{m,j}),$$
so that
\begin{align}\label{sqdrtttt}
\vert \tilde{\A}(x) - \A_{lin}(x) \vert \leq C \vert x- z_{m,j} \vert^2 \quad \text{on } \mathcal{B}_{m,j}.
\end{align}
Then if $\tilde{q}_{\hbar}$ denotes the quadratic form for the new potential $\tilde{\A}$, for $v \in \mathcal{C}_0^{\infty}(\mathcal{B}_{m,j})$,
$$\tilde{q}_{\hbar}(v) = q^{lin}_{\hbar}(v) + \Vert (\tilde{\A} - \A_{lin})v \Vert^2 + 2 \Re \langle (\tilde{\A} - \A_{lin})v , (i\hbar \nabla + \A_{lin})v \rangle,$$
and using (\ref{sqdrtttt}) and the Cauchy-Schwarz inequality,
$$\tilde{q}_{\hbar}(v) \geq q_{\hbar}^{lin}(v) -2C \Vert  \vert x -z_{m,j} \vert^2 v \Vert \sqrt{q_{\hbar}^{lin}(v)} .$$
We use $2 \vert a b \vert \leq \varepsilon^2 a^2 + \varepsilon^{-2} b^2$ to get:
\begin{align*}
\tilde{q}_{\hbar}(v) &\geq  \left( 1- C \hbar^{2\beta} \right) q_{\hbar}^{lin}(v) - C \hbar^{-2\beta} \Vert  \vert x - z_{m,j} \vert^2 v \Vert^2
\end{align*}
and so
\begin{align}
\tilde{q}_{\hbar}(v) \geq \left( 1- C \hbar^{2\beta} \right) q_{\hbar}^{lin}(v) - C \hbar^{4 \alpha-2\beta} \Vert  v \Vert^2.
\end{align}
Changing $\A$ into $\tilde{\A}$ amounts to conjugate the magnetic Laplacian by $e^{i\hbar^{-1} \varphi_{j,m}}$, so:
$$\tilde{q}_{\hbar}(v) = q_{\hbar}(e^{i \hbar^{-1} \varphi_{j,m}}v).$$ Hence, for any $v \in \mathcal{C}^{\infty}_0(\mathcal{B}_{m,j})$,
\begin{align}\label{eq730519}
q_{\hbar}(v) \geq \left( 1- C \hbar^{2\beta} \right) q_{\hbar}^{lin}(e^{-i\hbar^{-1} \varphi_{j,m}}v) - C \hbar^{4 \alpha-2\beta} \Vert  v \Vert^2.
\end{align}
$q_{\hbar}^{lin}(v)$ is the quadratic form associated to a constant magnetic field operator. Now, we approximate the metric with a flat one:
\begin{align}
q_{\hbar}^{lin}(v) & = \sum_{k,l} \int \vert g(x) \vert^{1/2} g^{kl}(x) (i \hbar \partial_k v + A^{lin}_k v)\overline{(i\hbar \partial_l v + A_l^{lin} v)} \dd x\\
& \geq \left( 1 - C \hbar^{\alpha} \right)\sum_{k,l} \int \vert g(z_{m,j}) \vert^{1/2} g^{kl}(z_{m,j}) (i \hbar \partial_k v + A_k^{lin} v)\overline{(i\hbar \partial_l v + A_l^{lin} v)} \dd x,\\
& = (1-C \hbar^{\alpha}) q_{\hbar}^{flat}(v). \label{flatone0519}
\end{align}
Hence, from (\ref{eqstar0519}) and (\ref{eq700519}) we get:
\begin{align*}
q_{\hbar}(\psi) &\geq \sum_{m=0}^{m_0} \sum_{j=0}^J  q_{\hbar}(\chi_{m,j}^{\hbar} \psi)  - \sum_{m=0}^{m_0} C \hbar^{2-2 \alpha} \Vert \chi_m \psi \Vert^2 - C \hbar^2 \Vert \psi \Vert^2 \\ &\quad + \sum_{m > m_0} \hbar (\tilde{b}_1-C_0 \hbar^{1/4}) \Vert \chi_m \psi \Vert^2,
\end{align*}
and using (\ref{eq730519}) and (\ref{flatone0519}):
\begin{align*}
q_{\hbar}(\psi) &\geq (1-C \hbar^{2 \beta})(1-C \hbar^{\alpha}) \sum_{m=0}^{m_0} \sum_{j=0}^J q_{\hbar}^{flat}(e^{-i \hbar^{-1} \varphi_{j,m}}\chi_{m,j}^{\hbar} \psi) \\ & \quad -  C \hbar^{4\alpha - 2 \beta} \sum_{m=0}^{m_0} \sum_{j=0}^J \Vert \chi_{m,j}^{\hbar} \psi \Vert^2  \\ & \quad  - (C \hbar^{2-2\alpha} + C \hbar^2 + C_0 \hbar^{5/4}) \Vert \psi \Vert^2 + \sum_{m > m_0} \hbar \tilde{b}_1 \Vert \chi_m \psi \Vert^2\\
& \geq (1-C \hbar^{2 \beta})(1-C \hbar^{\alpha}) \sum_{m=0}^{m_0} \sum_{j=0}^J  q_{\hbar}^{flat}(e^{-i \hbar^{-1} \varphi_{j,m}}\chi_{m,j}^{\hbar} \psi) + \sum_{m > m_0} \hbar \tilde{b}_1 \Vert \chi_m \psi \Vert^2 \\ & \quad  - ( C \hbar^{4 \alpha - 2 \beta} + C \hbar^{2-2\alpha} + C \hbar^2 + C_0 \hbar^{5/4}) \Vert \psi \Vert^2.
\end{align*}
Then
\begin{align}\label{eq02051900}
q_{\hbar}(\psi) &\geq  (1-C \hbar^{2 \beta})(1-C \hbar^{\alpha}) q_{\mathcal{L}} \left( (\chi_{j,m}\psi)_{j,m}, (\chi_m \psi)_{m> m_0} \right) - K(\hbar) \Vert \psi \Vert^2,
\end{align}
where $K(\hbar) = C \hbar^{4 \alpha - 2 \beta} + C \hbar^{2-2\alpha} + C \hbar^2 + C_0 \hbar^{5/4}$ and
\begin{align*}
q_{\mathcal{L}}(&(\psi_{j,m})_{j,m}, (\psi_m)_{m > m_0}) \\ &:= 
  \sum_{m=0}^{m_0} \sum_{j=0}^J  q_{\hbar}^{flat}(e^{i \hbar^{-1} \varphi_{j,m}}\psi_{j,m}) + (1-C \hbar^{2\beta})^{-1}(1-C \hbar^{\alpha})^{-1} \sum_{m > m_0} \hbar \tilde{b}_1 \Vert \psi_m \Vert^2
\end{align*}
is the quadratic form associated to $$\mathcal{L} = \left( \bigoplus_{m=0}^{m_0} \bigoplus_{j=0}^J \mathcal{L}_{m,j} \right) \oplus \left( \bigoplus_{m > m_0} \mathcal{L}_m \right),$$ where $\mathcal{L}_{m,j}$ is a Schrödinger operator with constant magnetic field acting on $\Ld(\mathcal{B}_{m,j})$, and $\mathcal{L}_m$ is the multiplication by $(1-C \hbar^{2\beta})^{-1}(1-C \hbar^{\alpha})^{-1} \hbar \tilde{b}_1$ acting on $\Ld(\mathcal{V}_m)$. 

We test inequality (\ref{eq02051900}) on the $N(\Lh, b_1 \hbar)$-dimensional space $V$ spanned by the $N(\Lh, b_1 \hbar)$ first eigenfunctions of $\Lh$ (Corresponding to eigenvalues $\leq \hbar b_1$). For $\psi \in V$,
\begin{small}
$$(1-C \hbar^{2 \beta})(1-C \hbar^{\alpha}) q_{\mathcal{L}}((\chi_{j,m} \psi)_{j,m}, (\chi_m \psi)_{m>m_0}) \leq  \left(  b_1 \hbar + K(\hbar) \right) \Vert \psi \Vert^2.$$
\end{small}
 Then, since $$\psi \mapsto ((\chi_{j,m}\psi)_{0 \leq m \leq m_0, 0 \leq j \leq J} , (\chi_m \psi)_{m > m_0})$$ is one-to-one, the space
$$ \left\lbrace ((\chi_{j,m} \psi)_{j,m},(\chi_m \psi)_m) \in \left( \bigoplus_{j,m} \Ld(\mathcal{B}_{m,j}) \right) \oplus \left( \bigoplus_{m>m_0} \Ld(\mathcal{V}_m) \right) ; \psi \in V \right\rbrace$$
is $N(\Lh, b_1 \hbar)$-dimensional, and the min-max principle yields to:
\begin{small}
$$N(\Lh, b_1 \hbar) \leq N \left( \mathcal{L}, (\hbar b_1 + K(\hbar))(1-C \hbar^{2\beta})^{-1}(1-C \hbar^{\alpha})^{-1} \right).$$
\end{small}
Since $\mathcal{L}_{m,j}$ is a magnetic Laplacian with constant magnetic field, we know that, for $\hbar$ small enough:
$$N(\mathcal{L}_{m,j} , \grandO(\hbar)) = \grandO(\hbar^{-d/2}), \quad 0 \leq m \leq m_0, \quad 0 \leq j \leq J,$$
and
$$N(\mathcal{L}_m, \hbar b_1 + o(\hbar)) = 0, \quad m > m_0.$$
With $\alpha = 3/8$ and $\beta = 1/8$, $K(\hbar)=o(\hbar)$, so we deduce:
$$N(\Lh, \hbar b_1) = \grandO(\hbar^{-d/2}).$$
\end{proof}

The same result holds for $\BNh$:

\begin{lemma}\label{rankBNh} Let $b_1 \in (0,\tilde{b}_1)$. There exists $C>0$ and $\hbar_0 >0$ such that
$$ \text{for all } \hbar \in (0, \hbar_0), \quad N(\BNh,  \hbar b_1) \leq C \hbar^{-d/2}.$$
\end{lemma}

\begin{proof}
By Lemma \ref{perturbationH0}, we have:
$$\langle \BNh \psi, \psi \rangle \geq (1-\zeta) \langle \Lh^0 \psi,\psi \rangle \geq (1-\zeta) \hbar \langle \Bh \psi, \psi \rangle,$$
with $\Bh = \Op(\hat{b})$.
Using the min-max principle, it follows that
$$N( \BNh , \hbar b_1) \leq N ( \Bh, (1-\zeta)^{-1} b_1),$$
and using Weyl estimates (\cite{Sjo} Chapter 9, or \cite{Ker}), we get
$$N(\Bh, (1-\zeta)^{-1} b_1) = \grandO( \hbar^{-d/2}).$$
\end{proof}

\section{\textbf{Comparison of the spectra of} $\Lh$ \textbf{and} $\BNh$}

\subsection{Proof of Theorem \ref{comparisonOftheEigenvalues}}

We denote $$\lambda_1(\hbar) \leq \lambda_2(\hbar) \leq ...$$ the smallest eigenvalues of $\Lh$ and $$\nu_1(\hbar) \leq \nu_2(\hbar) \leq ...$$ the smallest eigenvalues of $\BNh$. The goal of this section is to prove the following theorem, using the results of section \ref{SectionMicrolocalisation}.

\begin{theorem}\label{spectralcomparison}
If $b_1<\tilde{b}_1$ and $\delta \in (0,1/2)$, then
$$\lambda_n(\hbar) = \nu_n(\hbar) + \grandO(\hbar^{\delta r}),$$
uniformly in $n$ such that $\lambda_n(\hbar) \leq \hbar b_1$ and $\nu_n(\hbar) \leq \hbar b_1$.
\end{theorem}

Together with Theorem \ref{wellexpansionBNh}, this theorem concludes the proofs of Theorems \ref{comparisonOftheEigenvalues} and \ref{expansionOftheEigenvalues}.

\begin{proof}
We will prove that $\nu_n(\hbar) \leq \lambda_n(\hbar) + \grandO(\hbar^{\delta r})$, the other inequality being similar. Let $1 \leq n \leq N(\Lh,\hbar b_1)$, and let us denote $\psi_{1,\hbar}, ... , \psi_{n,\hbar}$ the normalized eigenfunctions associated to the first eigenvalues of $\Lh$. We also denote $$V_{n,\hbar} = \mathsf{span} \lbrace \chi_1(\hbar^{-2 \delta} \Lh) \chi_0(q) \psi_{j,\hbar} : 1 \leq j \leq n \rbrace,$$ where $\chi_0$ and $\chi_1$ are defined in Theorem \ref{microloc}. We have the normal form:
\begin{align}\label{eqnormalform0405}
\Uht^* \Lh \Uht = \BNh + R_{\hbar}, 
\end{align}
where
\begin{align*}
\Uht = V_{\hbar} \Uh, \quad \text{is given by (\ref{quantifsymplecto}) and Theorem \ref{BNh}.}
\end{align*}
We will use the min-max principle. For $\psi \in \mathsf{span}_{1 \leq j \leq n} \psi_{j,\hbar},$ we denote $$\tilde{\psi} = \chi_1(\hbar^{-2 \delta} \Lh) \chi_0(q) \psi \quad \in V_{n,\hbar}$$
Such a $\psit$ is microlocalized on $ \Omega_{\hbar} \subset U \subset T^*M,$ where
$$\Omega_{\hbar} = \lbrace (q,p) \in T^*M : \vert p - A(q) \vert^2 < c \hbar^{2\delta}, q \in \Omega \rbrace.$$ (Indeed, the symbol of $\chi_1(\hbar^{-2 \delta} \Lh)$ is $\grandO(\hbar^{\infty})$ where $\chi_1(\hbar^{-2\delta} \vert p - A(q) \vert^2) \equiv 0$). Thus, since $V_{\hbar} V_{\hbar}^* = I$ microlocally on $U$ (\ref{microinv2}) and $U_{\hbar}$ is unitary, we deduce from (\ref{eqnormalform0405}) that:
\begin{align}\label{formulaBNhUhtPsit}
\langle \BNh \Uht^{*} \psit, \Uht^{*} \psit \rangle &= \langle \Lh \psit,  \psit \rangle - \langle R_{\hbar} \Uht^* \tilde{\psi}, \Uht^* \tilde{\psi} \rangle + \grandO(\hbar^{\infty})\Vert \psit \Vert^2,
\end{align}
On the first hand, by Theorem \ref{microloc}, we can change $\psit$ into $\psi$ up to an error of order $\hbar^{\infty}$. Indeed, by Lemma \ref{rank}, the estimates of Theorem \ref{microloc} remain true for $\psi$. We get:
\begin{align*}
\langle \Lh \psit , \psit \rangle = \langle \Lh \psi, \psi \rangle + \grandO(\hbar^{\infty}) \Vert \psi \Vert^2 \leq (\lambda_n(\hbar)  + \grandO(\hbar^{\infty})) \Vert \psi \Vert^2.
\end{align*}
On the other hand, the remainder is:
\begin{align*}
\langle R_{\hbar} \Uht^* \tilde{\psi}, \Uht^* \tilde{\psi} \rangle = \langle \Uh R_{\hbar} \Uh^* V_{\hbar}^* \psit, V_{\hbar}^* \psit \rangle.
\end{align*}
The function $V_{\hbar}^* \psit$ is microlocalized in $$\mathcal{V}_{\hbar} = \lbrace
(w,z) : w \in V, \vert z \vert^2 \leq c \hbar^{2\delta} \rbrace,$$
because $V_{\hbar}$ is a Fourier integral operator with phase function associated to the canonical transformation $\Phi$, which is sending $\Omega_{\hbar}$ (where $\psit$ is microlocalized) on $\mathcal{V}_{\hbar}$. 
Moreover, the symbol of the pseudo-differential operator $\Uh \Rh \Uh^*$ on $V$ is $\grandO( (\hbar + \vert z \vert^2)^{r/2})$ (Theorem \ref{BNh}), so we get:
$$\Uh R_{\hbar} \Uh^* V_{\hbar}^* \psit = \grandO(\hbar^{ \delta r}).$$
Thus equation (\ref{formulaBNhUhtPsit}) yields to:
$$ \langle \BNh \Uht^{*} \psit, \Uht^{*} \psit \rangle \leq ( \lambda_n(\hbar) + \grandO(\hbar^{\delta r}) )\Vert \Uht^* \psit \Vert^2,$$
for all $\psit \in V_{n,\hbar}$. Since $V_{n,\hbar}$ is $n$-dimensional, the min-max principle gives
$$\nu_n(\hbar) \leq \lambda_n(\hbar) + \grandO(\hbar^{\delta r}).$$
The same arguments give the opposite inequality, replacing Theorem \ref{microloc} and Lemma \ref{rank} by Theorem \ref{microloc2} and Lemma \ref{rankBNh}.
\end{proof}

\subsection{Proof of Corollary \ref{WeylEstimates}}

Let us prove the Weyl estimates stated in Corollary \ref{WeylEstimates}. The proof relies on the classical Weyl asymptotics for pseudo-differential operators with elliptic principal symbol (\cite{Sjo} Chapter 9, \cite{Ker} Appendix).
Let us first prove the Weyl estimates for the Normal form. For any $n \in \N^{d/2}$, $\BNh^{(n)}$ is a pseudo-differential operator with principal symbol $$\hbar \hat{b}^{[n]}(w)= \hbar \sum_{j=1}^{d/2} (2n_j+1) \widehat{\beta}_j(w).$$ Note that $$V_n := \lbrace \hat{b}^{[n]}(w) \leq b_1 \rbrace$$
is empty for all but finitely many $n$. For these $n$, the G$\overset{\circ}{\text{a}}$rding inequality gives
$$\langle \BNh^{(n)} \psi, \psi \rangle \geq \hbar (b_1 - c \hbar) \Vert \psi \Vert, \quad \forall \psi \in \mathcal{S}(\R^{d/2}),$$
so that
$$N( \BNh^{(n)}, b_1 \hbar) = N( \frac{1}{\hbar} \BNh^{(n)}, [b_1 - c \hbar, b_1])$$
which is $o(\hbar^{-d/2})$ by the classical Weyl asymptotics. For the other finitely many $n$,
$$V_n \subset \lbrace \hat{b}(w) \leq b_1 \rbrace$$
is a compact set with positive volume and thus the classical Weyl asymptotics gives
$$N \left( \BNh^{(n)} , b_1 \hbar \right) = N \left( \frac{1}{\hbar} \BNh^{(n)} , b_1 \right) \sim \frac{1}{(2\pi \hbar)^{d/2}} \mathrm{Vol} \left( V_n \right).$$
Using $$\spectre (\BNh) = \bigcup_n \spectre (\BNh^{(n)}),$$ we deduce that
$$N( \BNh , b_1 \hbar) \sim \frac{1}{(2\pi \hbar)^{d/2}} \sum_n \mathrm{Vol} \left( V_n \right).$$
Moreover, 
$$\mathrm{Vol}(V_n) = \int_{V_n} \dd y \dd \eta = \int_{\varphi^{-1}(V_n)} \varphi^*( \dd y \dd \eta),$$
where $\varphi$ is defined in Theorem \ref{symplectomorphism}. Since $\varphi$ is a symplectomorphism, we have $$B = \varphi^*( \dd \eta \wedge \dd y )$$ and thus
$$\frac{B^{d/2}}{(d/2)!} = \frac{1}{(d/2)!} \varphi^*( (\dd \eta \wedge \dd y) ^{d/2}) = \varphi^*(\dd y \dd \eta).$$
Hence
$$\mathrm{Vol}(V_n) = \int_{b^{[n]}(q) \leq b_1} \frac{B^{d/2}}{(d/2)!},$$
so that
$$N(\BNh, b_1 \hbar) \sim \frac{1}{(2\pi \hbar)^{d/2}} \sum_{n \in \N^{d/2}} \int_{b^{[n]}(q) \leq b_1} \frac{B^{d/2}}{(d/2)!},$$
where the sum is finite. It remains to compare $$N_1 := N(\BNh, b_1 \hbar) \quad \text{and} \quad N_2 := N(\Lh, b_1 \hbar).$$ If we apply Theorem \ref{comparisonOftheEigenvalues} with some $b_1 + \delta > b_1$, we get a $c>0$ such that for $\hbar$ small enough,
$$ N( \BNh, \hbar b_1 - c \hbar^{r/2- \varepsilon}) \leq N_2 \leq N(\BNh, \hbar b_1 + c \hbar^{r/2-\varepsilon}),$$
so:
$$\vert N_1 - N_2 \vert \leq N( \BNh, [\hbar b_1 -c \hbar^{r/2- \varepsilon}, \hbar b_1 + c \hbar^{r/2- \varepsilon}]).$$
Classical Weyl asymptotics gives
$$N(\BNh^{(n)}, [\hbar b_1 -c \hbar^{r/2- \varepsilon}, \hbar b_1 + c \hbar^{r/2- \varepsilon}]) = o(\hbar^{-d/2}),$$
for any $n \in \N^{d/2}$, so $\vert N_1 - N_2 \vert = o(\hbar^{-d/2})$, and the proof is complete.

\section{\textbf{The case} $r_0 = \infty$}

If $r_0 = \infty$ (where $r_0$ is defined in (\ref{defrzero})), there is no resonances:
\begin{align}
\sum_{j=1}^{d/2} \alpha_j \beta_j(q_0) \neq 0, \quad \forall \alpha \in \N^{d/2}, \alpha \neq 0.
\end{align}
Of course, we can take any \textit{finite} $r \geq 3$, and construct the corresponding normal form. From Theorem \ref{expansionOftheEigenvalues} we deduce that
$$\forall r \geq 4, \forall j \geq 1, \lambda_j(\hbar) = \hbar b_0 +  \sum_{k=4}^{r-1} c_{j,k} \hbar^{k/2} + \grandO(\hbar^{r/2}),$$
so we get a complete expansion of $\lambda_j(\hbar)$ in powers of $\hbar^{1/2}$. However, the normal form depends on $r$. A natural question is : Could we construct a normal form which does not depend on $r$ ? The answer is yes, but we need to restrict to lower energies. Let us describe this construction. 

The reduction of the classical Hamiltonian does not depend on $r$, so there is nothing to change. The first problem appear with the formal normal form (Theorem \ref{algtheorem}). The problem is that the neighborhood $V$ on which the normal form is valid must reduce as $r$ goes to infinity. So we slightly change our definition of the space of formal series $\grandO_N (N \geq 0)$. Since the degree of a formal series $$\tau \in \mathcal{E} = \mathcal{C}^{\infty}(\R^d_w)[[x,\xi,\hbar]]$$
depend on $w$, we define $\grandO_N$ to be the set of formal series with valuation at least $N$ \textit{on a neighborhood of} 0. Then this neighborhood might go to zero as $N$ grows. Then the proof of Theorem (\ref{algtheorem}) remains true for $r= \infty$, and we get:

\begin{theorem}\label{algtheorem2}
If $\gamma \in \grandO_3$, there exist $\tau, \kappa \in \grandO_3$ such that:\\

$\bullet \quad e^{\frac{i}{\hbar} \ad_{\tau}} ( H^0 + \gamma )= H^0 + \kappa,$\\

$\bullet \quad [\kappa, \vert z_j \vert^2] = 0 \quad \text{for } 1 \leq j \leq d/2.$\\
\end{theorem}

Then we can quantize this result exactly as in Theorem \ref{BNh}, and we get:

\begin{theorem}
For $\hbar \in (0,\hbar_0]$ small enough, there exist a unitary operator $$\Uh : \Ld (\R^d ) \rightarrow \Ld (\R^d),$$ a smooth function $f^{\star}(w,I_1,...,I_{d/2},\hbar)$, and a pseudodifferential operator $\Rh$ such that:
\begin{align*}
&(i)\quad \Uh^* \Lhc \Uh = \Lh^0 + \Op f^{\star}(w, \Ih^{(1)}, ..., \Ih^{(d/2)}, \hbar) + \Rh,\\
&(ii)\quad f^{\star} \text{ has an arbitrarily small compact } (I_1,...,I_{d/2},\hbar)\text{-support (containing 0),}\\
&(iii)\quad \forall N \geq 3, \quad \symbolseries (\Rh) \in \grandO_N \quad \text{and} \quad \symbolseries (\Uh \Rh \Uh^*) \in \grandO_N.
\end{align*}
with $\Ih^{(j)} = \Op( \vert z_j \vert^2 )$ and $\Lh^0 = \Op ( H^0)$. We call $$\BNh = \Lh^0 + \Op f^{\star}(w, \Ih^{(1)}, ..., \Ih^{(d/2)}, \hbar)$$
the normal form, and $\Rh$ the remainder.
\end{theorem}

Moreover, up to replacing $f^{\star}$ by $\chi(\hbar^{-1} \ . \ ) f^{\star}$, (which does not change the properties of the normal form because $f^{\star}$ is defined by its Taylor series), we can adapt the proof of Proposition \ref{perturbationH0} to get

\begin{lemma}
We can construct the normal form $\BNh$ such that, for $\hbar \in (0,\hbar_0]$ small enough and some $C>0$:
$$(1-C\hbar) \langle \Lh^0 \psi, \psi \rangle \leq \langle \BNh \psi , \psi \rangle \leq (1+C \hbar) \langle \Lh^0 \psi, \psi \rangle, \quad \forall \psi \in \mathcal{S}(\R^d).$$
\end{lemma}

It remains to prove the analog of Theorem \ref{comparisonOftheEigenvalues}. For this, we need the following microlocalization results. Their proofs follow the same lines as in section 6. Note the retriction to energies $\lambda \leq \hbar (b_0 + c \hbar^{\eta})$, necessary to localize in a neighborhood of $q_0$ of decreasing size as $\hbar \rightarrow 0$. We define, for any fixed $c>0$:
\begin{align}
K_{\hbar} := \lbrace q \in M : b(q) \leq b_0 + 2c\hbar^{\eta} \rbrace,
\end{align}
and its small neighborhood
\begin{align}
K_{0,\hbar} := \lbrace q \in M : \dd (q, K_{\hbar}) \leq \hbar^{\eta} \rbrace.
\end{align}

\begin{theorem} \label{microlocInfiniteCase} Let $\delta \in (0,\frac{1}{2})$, $c>0$, and $\eta \in (0,1/4)$. Let $\chi_{\hbar} : M \rightarrow [0,1]$ be a smooth cutoff function being $1$ on $K_{0,\hbar}$. Let $\chi_1 : \R \rightarrow [0,1]$ be a smooth cutoff function being $1$ near $0$. Then for any normalized eigenpair $(\lambda, \psi)$ of $\Lh$ such that $\lambda \leq \hbar (b_0 + c \hbar^{\eta})$ we have:
$$ \psi = \chi_1(\hbar^{-2\delta} \Lh) \chi_{\hbar}(q) \psi + \grandO(\hbar^{\infty}) \quad \text{in } \Ld(M),$$
uniformly with respect to $(\lambda,\psi)$.
\end{theorem}

The proof follows the same lines as Theorem \ref{Agmon}, with $\alpha = 1/4$, $K$ replaced by $K_{\hbar}$, $K_{\varepsilon}$ replaced by $K_{0,\hbar}$, and Theorem \ref{microloc} with no change. The uniformity with respect to $(\lambda,\psi)$ follows from Lemma \ref{rank}.

Similarly, we have the microlocalization Theorem for the normal form $\BNh$. We denote
$$V_{\hbar} := \lbrace w \in \R^d, \dd (w, \varphi(K_{0,\hbar})) < \hbar \rbrace.$$

\begin{theorem} \label{microloc2InfiniteCase} Let $\hbar \in (0,\hbar_0]$, $c>0$, $\eta \in (0,1/4)$ and $\delta \in (0,\eta/2)$. Let $\chi_{0}$ be a smooth cutoff function on $\R^{d/2}_w$ supported on $V$ such that $\chi_0=1$ near $0$ and $\chi_1$ a smooth cutoff function being $1$ near $0$. Then for any normalized eigenpair $(\lambda, \psi)$ of $\BNh$ such that $\lambda \leq \hbar (b_0 + c \hbar^{\eta})$, we have:
$$ \psi = \chi_1(\hbar^{-2\delta} \Ih^{(1)})...\chi_1 (\hbar^{-2\delta} \Ih^{(d/2)}) \Op (\chi_{0}(\hbar^{-\delta}w)) \psi  + \grandO(\hbar^{\infty}) \quad \text{in } \Ld(\R^d),$$
uniformly with respect to $(\lambda, \psi)$.
\end{theorem}

\begin{proof} We follow the proof of Lemma \ref{microBNh1}. With $\chi(w) = 1- \chi_0(\hbar^{-\delta} w)$, Inequality (\ref{01051902}) becomes
\begin{align*}
\langle \BNh \Op (\chi ) \psi, \Op( \chi ) \psi \rangle \leq \hbar (b_0 + c \hbar^{\eta}) \Vert \Op ( \chi ) \psi \Vert^2 + \langle [\BNh, \Op( \chi) ] \psi , \Op (\chi) \psi \rangle,
\end{align*}
And the estimate (\ref{01051901}) on the commutator becomes
\begin{align*}
\langle [ \BNh, \Op ( \chi ) ] \psi, \Op( \chi )\psi \rangle \leq \hbar^{2- \delta}\Vert \Op ( \bar{\chi} ) \psi \Vert^2,
\end{align*}
because the commutator is of order $\hbar^{1-\delta}$. The lower bound becomes
\begin{align*}
\langle \BNh \Op(\chi) \psi, \Op(\chi) \psi \rangle &\geq (1-C\hbar) \langle \Lh^0 \Op (\chi) \psi, \Op(\chi) \psi \rangle\\
& \geq (1- C\hbar) \hbar (b_0 + \tilde{C}\hbar^{2\delta}) \Vert \Op ( \chi ) \psi \Vert^2.
\end{align*}
Hence we get
\begin{align*}
\left[(1-C \hbar)(b_0 + \tilde{C}\hbar^{2\delta}) - (b_0 + c \hbar^{\eta}) \right] \Vert \Op ( \chi ) \psi \Vert^2 \leq \hbar^{1-\delta} \Vert \Op (\bar{\chi} ) \psi \Vert^2.
\end{align*}
Since $2\delta < \eta$, we get a new $C>0$ such that for $\hbar$ small enough:
\begin{align*}
C \hbar^{2 \delta} \Vert \Op ( \chi ) \psi \Vert^2 \leq \hbar^{1-\delta} \Vert \Op (\bar{\chi} ) \psi \Vert^2.
\end{align*}
Iterating with $\bar{\chi}$ instead of $\chi$, for $\delta < 1/3$ we get
$$\Op(\chi) \psi = \grandO(\hbar^{\infty}).$$
The end of the proof is the same as the proof of Theorem \ref{microloc2}. The uniformity with respect to $(\lambda,\psi)$ comes from Lemma \ref{rankBNh}.
\end{proof}

Since the eigenfunctions of $\BNh$ and $\Lh$ are microlocalized on a neighborhood of the minimum of diameter going to $0$ as $\hbar \rightarrow 0$, we can follow the proof of Theorem \ref{comparisonOftheEigenvalues} (section 7) to get:

\begin{theorem}
Let $c>0$ and $\eta \in (0,1/4)$. We denote $$\lambda_1(\hbar) \leq \lambda_2(\hbar) \leq ... \quad \text{and} \quad \nu_1(\hbar) \leq \nu_2(\hbar) \leq ...$$ the first eigenvalues of $\Lh$ and $\BNh$. Then
$$\lambda_n(\hbar) = \nu_n(\hbar) + \grandO(\hbar^{\infty}),$$
uniformly in $n$ such that $\lambda_n(\hbar) \leq \hbar (b_0 + c \hbar^{\eta})$ and $\nu_n(\hbar) \leq \hbar (b_0 + c \hbar^{\eta})$.
\end{theorem}

\appendix

\section{}

\begin{lemma}\label{SubprincipalSymbol}
The principal and subprincipal symbols of the operator
$$\Lh = (i\hbar \dd + A)^{*}(i\hbar \dd + A)$$
are
$$\sigma_0(\Lh) = \vert p-A(q) \vert^2_{g^*(q)}, \quad \text{and} \quad  \sigma_1(\Lh) = 0.$$
\end{lemma}

\begin{proof}
We will compute these symbols in coordinates, in which $\Lh$ acts as:
$$ \Lh^{coord} = \sum_{k \ell} \vert g \vert^{-1/2} (i \hbar \partial_k + A_k) g^{k \ell} \vert g \vert^{1/2} (i\hbar \partial_{\ell}+ A_{\ell}).$$
The principal symbol is always well-defined. The subprincipal symbol is well-defined if we restrict the changes of coordinates to be volume-preserving. This amounts to conjugating $\Lh^{coord}$ by $\vert g \vert^{1/4}$. Thus the subprincipal symbol is defined in coordinates by:
$$\sigma_1(\Lh) = \sigma_1(\vert g \vert^{1/4} \Lh^{coord} \vert g \vert^{-1/4}).$$
The total symbol of $-i\hbar \partial_k -A_k$ is
$$\sigma ( - i \hbar \partial_k - A_k ) = p_k - A_k,$$
so we can use the star product $\star$ on symbols to compute the symbol of $\Lh$:
$$\sigma(\vert g \vert^{1/4} \Lh^{coord} \vert g \vert^{-1/4}) = \sum_{k \ell} \vert g \vert^{1/4} \star \vert g \vert^{-1/2} \star (p_k - A_k)  \star  g^{k \ell} \vert g \vert^{1/2} \star (p_{\ell} - A_{\ell})  \star \vert g \vert^{-1/4}.$$
Now we will use the formula
$$\sigma ( f \star g ) = fg + \frac{\hbar}{2i} \lbrace f, g \rbrace + \grandO(\hbar^2)$$
several times to compute the symbol, where $\lbrace f,g \rbrace$ denotes the Poisson brackets. Of course, we directly deduce the principal symbol:
$$\sigma_0(\vert g \vert^{1/4} \Lh^{coord} \vert g \vert^{-1/4}) = \sum_{k \ell} g^{k \ell} (p_k - A_k)(p_{\ell}-A_{\ell})$$
so that
$$\sigma_0 ( \Lh) = \vert p -A(q) \vert^2_{g^*(q)}.$$
To compute the subprincipal symbol, we will use:
\begin{small}
$$\sigma (\vert g \vert^{1/4} \Lh^{coord} \vert g \vert^{-1/4}) = \sum_{k \ell} \left[ \vert g \vert^{-1/4} \star (p_k - A_k) \star \vert g \vert^{1/4} \right] \star g^{k\ell} \star \left[ \vert g \vert^{1/4} \star (p_{\ell} - A_{\ell}) \star \vert g \vert^{-1/4} \right].$$
\end{small}
Let us compute $a_k = \vert g \vert^{-1/4} \star (p_k - A_k) \star \vert g \vert^{1/4}$.
\begin{align*}
a_k &= (p_k - A_k) + \frac{\hbar}{2i} \left[ \lbrace \vert g \vert^{-1/4} (p_k - A_k), \vert g \vert^{1/4} \rbrace + \lbrace \vert g \vert^{-1/4} , p_k - A_k \rbrace \vert g \vert^{1/4} \right] + \grandO(\hbar^2)\\
&= (p_k - A_k) + \frac{\hbar}{2i} \left[ \vert g \vert^{-1/4} \frac{\partial \vert g \vert^{1/4}}{\partial q_k} - \frac{\partial \vert g \vert^{-1/4}}{\partial q_k} \vert g \vert^{1/4} \right] + \grandO(\hbar^2) \\
&=(p_k - A_k) + \frac{\hbar}{i} \vert g \vert^{-1/4} \frac{\partial \vert g \vert^{1/4}}{\partial q_k} + \grandO(\hbar^2).
\end{align*}
We also get the similar result for $b_{\ell} = \vert g \vert^{1/4} \star (p_{\ell} - A_{\ell}) \star \vert g \vert^{-1/4}$:
\begin{align*}
b_{\ell} = (p_{\ell} - A_{\ell}) - \frac{\hbar}{i} \vert g \vert^{-1/4} \frac{\partial \vert g \vert^{1/4} }{\partial q_{\ell}} + \grandO(\hbar^2)
\end{align*}
Thus we can compute
\begin{align*}
a_k \star g^{k \ell} &= g^{k\ell} (p_k - A_k) + \frac{\hbar}{2i} \lbrace p_k - A_k , g^{k \ell} \rbrace + \frac{\hbar}{i} \vert g \vert^{-1/4} \frac{\partial \vert g \vert^{1/4}}{\partial q_k} g^{k\ell} + \grandO(\hbar^2)\\
&= g^{k \ell} (p_k - A_k) + \frac{\hbar}{2i} \frac{\partial g^{k \ell}}{\partial q_k} + \frac{\hbar}{i} \vert g \vert^{-1/4} \frac{\partial \vert g \vert^{1/4}}{\partial q_k} g^{k\ell}+ \grandO(\hbar^2),\\
\end{align*}
and
\begin{small}
\begin{align*}
a_k \star g^{k \ell} \star b_{\ell} &= g^{k \ell} (p_k - A_k) (p_l - A_l) + \frac{\hbar}{2i} \lbrace g^{k \ell} (p_k - A_k) , p_{\ell} - A_{\ell}\rbrace - \frac{\hbar}{i} g^{k \ell} (p_k - A_k) \vert g \vert^{-1/4} \frac{\partial \vert g \vert^{1/4}}{\partial q_{\ell}}\\
& \quad  + \frac{\hbar}{2i} \frac{\partial g^{k \ell}}{\partial q_k} (p_{\ell} - A_{\ell}) + \frac{\hbar}{i} \vert g \vert^{-1/4} \frac{\partial \vert g \vert^{1/4}}{\partial q_k} (p_{\ell}-A_{\ell}) + \grandO(\hbar^2).
\end{align*}
\end{small}
Summing over $k,\ell$, we get
\begin{small}
\begin{align*}
\sum_{k \ell} a_k \star g^{k \ell} \star b_{\ell} &= \sum_{k \ell} g^{k \ell} (p_k - A_k) (p_l - A_l) + \frac{\hbar}{2i} \lbrace g^{k \ell} (p_k - A_k) , p_{\ell}- A_{\ell} \rbrace \\ 
& \quad + \frac{\hbar}{2i} \frac{\partial g^{k \ell}}{\partial q_k} (p_{\ell} - A_{\ell}) + \grandO(\hbar^2)\\
&= \sum_{k \ell} g^{k \ell} (p_k - A_k)(p_{\ell} - A_{\ell}) + \frac{\hbar}{2i}g^{k \ell} \frac{\partial(p_{\ell} - A_{\ell})}{\partial q_{k}} - \frac{\hbar}{2i} \frac{\partial g^{k \ell} (p_k - A_k)}{\partial q_{\ell}} \\
& \quad + \frac{\hbar}{2i} \frac{\partial g^{k \ell}}{\partial q_k} (p_{\ell} - A_{\ell}) + \grandO(\hbar^2)\\
&= \sum_{k \ell} g^{k \ell} (p_k - A_k)(p_{\ell}- A_{\ell}) + \grandO(\hbar^2).
\end{align*}
\end{small}
Since $$\sigma(\vert g \vert^{1/4} \Lh^{coord} \vert g \vert^{-1/4}) = \sum_{k \ell} a_k \star g^{k \ell} \star b_{\ell},$$ we deduce that: $$\sigma_1(\vert g \vert^{1/4} \Lh^{coord} \vert g \vert^{-1/4} ) = 0,$$
and
$$\sigma_1( \Lh) = 0.$$
\end{proof}

The following Lemma due to Weinstein \cite{Weinstein} tells that, if two 2-forms coincide on a submanifold, they are equal up to a transformation tangent to the identity.

\begin{lemma}[Relative Darboux lemma] \label{RelativeDarboux} Let $\omega_0$ and $\omega_1$ be two $2$-forms on $\Omega \times \R^d_z$ which are closed and non degenerate. Assume that $\omega_{0 \vert z=0} = \omega_{1 \vert z=0}$. Then there exists a change of coordinates $S$ on a neighborhood of $\Omega \times \lbrace 0 \rbrace$ such that $$S^* \omega_1 = \omega_0 \quad \text{and} \quad S = Id + \grandO(\vert z \vert^2).$$
\end{lemma}

For a proof, see for example \cite{article} and the references therein. The next Lemma states the Agmon formula (see \cite{Agmon}).

\begin{lemma}[Agmon formula]\label{AgmonFormula}
Let $\psi$ be an eigenfunction of $\Lh$ associated to $\lambda$, and $\Phi : M \rightarrow \R$ is a Lipschitz function such that $e^{\Phi}\psi$ be in the domain of $q_{\hbar}$, then $\dd \Phi$ is defined almost everywhere and:
$$q_{\hbar}(e^{\Phi}\psi) = \lambda \Vert e^{\Phi}\psi \Vert^2 + \hbar^2 \Vert e^{\Phi}\psi \ \dd \Phi \Vert^2.$$
\end{lemma}

\begin{proof}
First note that:
$$q_{\hbar}(e^{\Phi}\psi) = \langle \Lh e^{\Phi}\psi, e^{\Phi}\psi \rangle_{\Ld(M)} = \lambda \Vert e^{\Phi}\psi \Vert^2 + \langle [\Lh, e^{\Phi}] \psi, e^{\Phi} \psi \rangle_{\Ld(M)},$$
so we need to compute the bracket.
\begin{align*}
\langle & [\Lh,e^{\Phi}] \psi, e^{\Phi} \psi \rangle \\
&= \int_M \langle (i\hbar\dd +A)^*(i\hbar \dd + A) e^{\Phi}\psi,e^{\Phi}\psi \rangle \dd q - \int_M \langle e^{\Phi} (i\hbar\dd +A)^*(i\hbar \dd + A) \psi, e^{\Phi} \psi \rangle \dd q\\
&= \int_M \vert (i\hbar \dd +A)e^{\Phi}\psi \vert^2 \dd q - \int_M \langle (i\hbar \dd + A) \psi, (i\hbar \dd + A)e^{2\Phi}\psi \rangle \dd q\\
\end{align*}
On the one hand,
\begin{align*}
\int_M \langle (i\hbar \dd + A) \psi, (i\hbar \dd + A)e^{2\Phi}\psi \rangle \dd q = \int_M \left( \vert e^{\Phi} ( i\hbar \dd + A ) \psi \vert^2  + 2e^{2\Phi} \langle (i\hbar \dd + A) \psi,i\hbar \psi \dd \Phi \rangle \right) \dd q,
\end{align*}
and taking the real part:
\begin{align*}
\int_M \langle (i\hbar \dd + A) \psi, (i\hbar \dd + A)e^{2\Phi}\psi \rangle \dd q = \int_M \left( \vert e^{\Phi} ( i\hbar \dd + A ) \psi \vert^2  + 2 \Re e^{2\Phi} \langle (i\hbar \dd + A) \psi,i\hbar \psi \dd \Phi \rangle \right) \dd q.
\end{align*}
On the other hand,
\begin{align*}
\int_M \vert (i\hbar \dd + A ) e^{\Phi} \psi \vert ^2 \dd q = \int_M  &\vert e^{\Phi} (i\hbar \dd + A) \psi \vert^2 + \vert i\hbar \psi e^{\Phi} \dd \Phi \vert^2\\ &+ 2\Re \langle e^{\Phi} (i \hbar \dd + A) \psi, i \hbar \psi e^{\Phi} \dd \Phi \rangle \dd q,
\end{align*}
so we finally get:
$$\langle [ \Lh,e^{\Phi} ] \psi , e^{\Phi} \psi \rangle = \hbar^2 \Vert e^{\Phi} \psi \dd \Phi \Vert^2.$$
\end{proof}

In \cite{HeMo96}, the following theorem is proved in the case $M$ is compact or the Euclidean $\R^d$. Here we just adapted their proof for non-compact manifolds, with a possible boundary.

\begin{lemma}\label{manifoldinequality}
 Assume that $(M,g)$ is either compact or with bounded variations in the following sense : There exists a compact subset $\mathcal{K} \subset M$ and finitely many charts $$\Psi_n : U_n \rightarrow \mathcal{V} \subset \R^d, \quad 1 \leq n \leq n_0,$$ with $$M \setminus \mathcal{K} = \bigcup_{n=1}^{n_0} U_n,$$ under which the Riemannian metric satisfies
$$\frac{\partial g^{ij}}{\partial x_k} \text{is bounded for } 1 \leq i,j,k \leq d,$$
and
$$\frac{\partial \vert g \vert^{1/2}}{\partial x_k} \text{is bounded for } 1 \leq k \leq d.$$ Then if $\B$ is such that
\begin{align}\label{AssB}
\vert \nabla \B(q) \vert \leq C(1+ \vert \B(q) \vert),
\end{align}
 there exists $\hbar_0 >0$ and $C_0 >0$ such that, for $\hbar \in (0, \hbar_0]$,
$$\forall u \in D(q_{\hbar}), \quad (1+\hbar^{1/4}C_0) q_{\hbar}(u) \geq \int_{M} \hbar (b(q) - \hbar^{1/4} C_0) \vert u(q) \vert^2 \dd q.$$
\end{lemma}

\begin{proof}
Take $(\chi_m)_{m \geq 0}$ a smooth partition of unity on $M$, such that:
$$\sum_{m \geq 0} \chi_m^2 = 1 \quad \text{and} \quad \sum_{m \geq 0} \vert \dd \chi_m(q) \vert ^2 \leq C, \quad \forall q \in M,$$
with $\supp (\chi_m) \subset \mathcal{V}_m$ a bounded local chart. Then by Lemma \ref{CalculQhPartitionUnity} (below), for any $u \in D(q^{\hbar})$,
$$q_{\hbar} (u) = \sum_{m \geq 0} q_{\hbar}(\chi_m u) - \hbar^2 \sum_{m \geq 0} \Vert  u \dd \chi_m \Vert^2 \geq \sum_{m \geq 0} q_{\hbar}(\chi_m u) - C \hbar^2 \Vert u \Vert^2,$$
and we can deal with every $q_{\hbar}(\chi_m u)$ in local coordinates $x=(x_1,...,x_d)$: we can write $q_{\hbar}(\chi_m u) = q_{\hbar}^{\text{coord}}(\tilde{\chi}_m\tilde{u})$, where $\tilde{u}$ stands for $u$ written in coordinates. We denote $\langle \B(x) \rangle = ( 1+ \vert \B(x) \vert^2 )^{1/2}$. Under assumption (\ref{AssB}), up to taking $\mathcal{V}_m$ small enough, we can find $z_m \in M$ and $C>0$ such that:
\begin{align}\label{qdrty}
C^{-1} \langle \B(x) \rangle \leq \langle \B(z_m) \rangle \leq C \langle \B(x) \rangle, \quad \forall x \in \mathcal{V}_m.
\end{align}
Indeed, by (\ref{AssB}), if we denote $M(\varepsilon) = \sup_{\vert x- y \vert \leq \varepsilon} \frac{\langle \B(x) \rangle}{\langle \B(y)\rangle}$ and $m(\varepsilon) = \inf_{\vert x- y \vert \leq \varepsilon} \frac{\langle \B(x) \rangle}{\langle \B(y)\rangle}$, we have:
$$\forall \vert x - y \vert \leq \varepsilon, \exists c_{xy} \in [x,y], \langle \B(x) \rangle \leq \langle \B(y) \rangle + C \langle \B(c_{xy}) \rangle \vert x - y \vert,$$
which implies $$M(\varepsilon) \leq 1 + C M(\varepsilon) \varepsilon,$$
and for $\varepsilon < 1/2C$, $$M(\varepsilon) \leq 2.$$ Similarily, we have
$$\frac{\langle \B(x) \rangle}{\langle \B(y) \rangle} \geq 1 - C \frac{\langle \B(c_{cy}) \rangle}{\langle \B(y) \rangle} \vert x - y \vert \geq 1 - C M(\varepsilon) \vert x - y \vert \geq 1 - 2C \varepsilon.$$
Rescaling a standard partition of unity on $\R^d$, we can find a new partition of unity $(\chi_{m,j}^{\hbar})_{j \geq 0}$ on $\mathcal{V}_m$ such that:
\begin{align}\label{qdrttt}
\sum_{j \geq 0} \vert \chi_{m,j}^{\hbar} \vert ^2 =1, \quad \text{and} \quad \sum_{j \geq 0} \vert \dd \chi_{m,j}^{\hbar}(x) \vert^2 \leq C \langle \B(z_m) \rangle \hbar^{-2 \alpha},
\end{align}
where $C>0$ does not depend on $m$, and with
\begin{align}\label{DefBmj}
\supp (\chi_{m,j}^{\hbar}) \subset \mathcal{B}_{m,j} := \lbrace x : \vert x-y_{m,j} \vert \leq \langle \B(z_{m}) \rangle ^{-1/2}\hbar^{\alpha} \rbrace.
\end{align}
Then for any $u \in \mathcal{C}^{\infty}_0(\mathcal{V}_m)$,
\begin{align}
q_{\hbar}(u) &\geq \sum_j q_{\hbar}(\chi_{m,j} u) - C \langle \B(z_m) \rangle \hbar^{2-2 \alpha} \Vert u \Vert^{2}\\
& \label{inequalityQhBmj}\geq \sum_j q_{\hbar}(\chi_{m,j} u) - C \hbar^{2-2 \alpha} \int (b(x) + 1) \vert u \vert^{2} \dd x_g,
\end{align}
because $\langle \B(z_m) \rangle \leq C \langle \B(x) \rangle \leq C' (b(x)+1)$.
Since $b$ is continuous, on each $\mathcal{B}_{m,j}$ we can choose $z_{m,j}$ such that 

\begin{align}\label{qsdminoration}
b(z_{m,j}) \geq b(x), \quad \forall x \in \mathcal{B}_{m,j}.
\end{align}
On each $\mathcal{B}_{m,j}$, we will approximate the magnetic field by a constant. Up to a gauge transformation, we can assume that the vector potential vanishes at $z_{m,j}$. In other words, we can find a smooth function $\varphi_{m,j}$ on $\mathcal{B}_{m,j}$ such that $$\tilde{\A}(z_{m,j})=0,$$
where $\tilde{\A} = \A + \nabla \varphi_{m,j}.$ The potential $\tilde{\A}$ defines the same magnetic field $\B$ as $\A$. Let us define
$$\A_{lin}(x) = \B (z_{m,j}). (x-z_{m,j}),$$
so that
$$\vert \tilde{\A}(x) - \A_{lin}(x) \vert \leq \frac{1}{2}\Vert \nabla \B \Vert_{\mathcal{B}_{j,m}} \vert x- z_{m,j} \vert^2, \quad \text{on } \mathcal{B}_{m,j},$$
and using (\ref{AssB}) and (\ref{qdrty}),
\begin{align}\label{sqdrtt}
\vert \tilde{\A}(x) - \A_{lin}(x) \vert  \leq C \langle \B(z_{m,j}) \rangle \vert x - z_{m,j} \vert^2, \quad \text{on } \mathcal{B}_{m,j}.
\end{align}
Then if $\tilde{q}_{\hbar}$ denotes the quadratic form for the new potential $\tilde{\A}$, for $v \in \mathcal{C}_0^{\infty}(\mathcal{B}_{m,j})$,
$$\tilde{q}_{\hbar}(v) = q^{lin}_{\hbar}(v) + \Vert (\tilde{\A} - \A_{lin})v \Vert^2 + 2 \Re \langle (\tilde{\A} - \A_{lin})v , (i\hbar \nabla + \A_{lin})v \rangle,$$
and using (\ref{sqdrtt}) and the Cauchy-Schwarz inequality,
$$\tilde{q}_{\hbar}(v) \geq q_{\hbar}^{lin}(v) -2C \Vert \langle \B(z_{m,j}) \rangle \vert x -z_{m,j} \vert^2 v \Vert \sqrt{q_{\hbar}^{lin}(v)} .$$
We use $2 \vert a b \vert \leq \varepsilon^2 a^2 + \varepsilon^{-2} b^2$ to get:
\begin{align*}
\tilde{q}_{\hbar}(v) &\geq  \left( 1- C \hbar^{2\beta} \right) q_{\hbar}^{lin}(v) - C \hbar^{-2\beta} \Vert \langle \B(z_{m,j} ) \rangle  \vert x - z_{m,j} \vert^2 v \Vert^2\\
& \geq \left( 1- C \hbar^{2\beta} \right) q_{\hbar}^{lin}(v) - \tilde{C} \hbar^{4 \alpha-2\beta} \Vert  v \Vert^2  \quad \text{by (\ref{qdrty}) and (\ref{DefBmj}).}
\end{align*}
Changing $\A$ into $\tilde{\A}$ amounts to conjugate the magnetic Laplacian by $e^{i\hbar^{-1} \varphi_{j,m}}$, so:
$$\tilde{q}_{\hbar}(v) = q_{\hbar}(e^{i \hbar^{-1} \varphi_{j,m}}v).$$ We get for $v \in \mathcal{C}_0^{\infty}(\mathcal{B}_{m,j})$:
\begin{align*}
q_{\hbar}(v) & \geq \left( 1- C \hbar^{2\beta} \right) q_{\hbar}^{lin}(e^{-i \hbar^{-1} \varphi_{j,m}}v) - \tilde{C} \hbar^{4 \alpha-2\beta} \Vert e^{i \hbar^{-1} \varphi_{j,m}} v \Vert^2\\
& \geq \left( 1- C \hbar^{2\beta} \right) q_{\hbar}^{lin}(e^{-i \hbar^{-1} \varphi_{j,m}}v) - \tilde{C} \hbar^{4 \alpha-2\beta} \Vert v \Vert^2.
\end{align*}
It remains to estimate $q_{\hbar}^{lin}(v)$. Using the assumptions on $M$,
\begin{align*}
q_{\hbar}^{lin}(v) & = \sum_{k,l} \int \vert g(x) \vert^{1/2} g^{kl}(x) (i \hbar \partial_k v + A^{lin}_k v)\overline{(i\hbar \partial_l v + A_l^{lin} v)} \dd x\\
& \geq \left( 1 - C \hbar^{\alpha} \right)\sum_{k,l} \int \vert g(z_{m,j}) \vert^{1/2} g^{kl}(z_{m,j}) (i \hbar \partial_k v + A_k^{lin} v)\overline{(i\hbar \partial_l v + A_l^{lin} v)} \dd x.\\
\end{align*}
For this new Schrödinger operator with constant magnetic field on a flat metric, the desired inequality is well known:
\begin{align*}
q_{\hbar}^{lin}(v) & \geq \left( 1 - C \hbar^{\alpha} \right) \hbar \int b(z_{m,j}) \vert v \vert^2 \vert g(z_{m,j}) \vert^{1/2} \dd x\\
& \geq \left( 1 - C \hbar^{\alpha} \right) \hbar \int b(x) \vert v \vert^2 \vert g(z_{m,j}) \vert^{1/2} \dd x\\
& \geq \left( 1 - C_1 \hbar^{\alpha} \right) \hbar \int b(x) \vert v \vert^2 \vert g(x) \vert^{1/2} \dd x.
\end{align*}
because of (\ref{qsdminoration}) and the assumptions on $M$. Thus,
\begin{align*}
q_{\hbar}^{lin}(e^{-i \hbar^{-1} \varphi_{m,j}}v) \geq \left( 1 - C_1 \hbar^{\alpha} \right) \hbar \int b(x) \vert v \vert^2 \vert g(x) \vert^{1/2} \dd x.
\end{align*}
 Finally, we get a $C_0 >0$ such that, for $\hbar$ small enough,
\begin{align*}
(1+ C_0 \hbar^{2 \beta} + C_1 \hbar^{\alpha}) q_{\hbar}(u) \geq & \ \hbar \int_M b(x) \vert u \vert^2 \dd x_g -\tilde{C}\hbar^{4 \alpha - 2 \beta} \Vert u \Vert^2 - C \hbar^2 \Vert u \Vert^2\\ &- C_0 \hbar^{2-2 \alpha} (\int_M b(x) \vert u \vert^2 \dd x_g + \Vert u \Vert^2),
\end{align*}
where the last part comes from (\ref{inequalityQhBmj}). The desired inequality follows if we choose $\beta = 1/8$ and $\alpha = 3/8$.
\end{proof}

\begin{lemma}\label{CalculQhPartitionUnity}
If $(\chi_m)_{m \geq 0}$ is a smooth partition of unity on $M$, such that
$$\sum_{m \geq 0} \chi_m^2 =1,$$
then for any $u \in D(q_{\hbar})$:
$$q_{\hbar}(u) = \sum_{m \geq 0} q_{\hbar}(\chi_m u) - \hbar^2 \sum_{m \geq 0} \Vert  u \dd \chi_m\Vert^2.$$
\end{lemma}

\begin{proof}
\begin{align*}
q_{\hbar}(u) &= \sum_m \int \vert \chi_m  (i \hbar \dd + A) u \vert^2 \dd q_g\\
&= \sum_m \int \vert (i \hbar \dd + A) (\chi_m u) - [i \hbar \dd + A, \chi_m] u \vert^2 \dd q_g\\
&= \sum_m \int \vert ( i \hbar \dd + A)(\chi_m u) - i \hbar u (\dd \chi_m) \vert^2 \dd q_g\\
&= \sum_m \int \vert ( i \hbar \dd + A)(\chi_m u) \vert^2 + \hbar^2 \vert u \dd \chi_m \vert^2 - 2 \Re \langle (i \hbar \dd + A)(\chi_m u), i \hbar u \dd \chi_m \rangle \dd q_g\\
&= \sum_m q_{\hbar}(\chi_m u) + \hbar^2 \Vert u \dd \chi_m \Vert^2- \int 2 \Re \langle (i \hbar \dd + A)(\chi_m u), i \hbar u \dd \chi_m \rangle \dd q_g.
\end{align*}
Moreover,
\begin{align*}
\langle (i \hbar \dd + A)(\chi_m u), i \hbar u \dd \chi_m \rangle &= \langle i \hbar u \dd \chi_m + i \hbar \chi_m \dd u + \chi_m u A, i \hbar u \dd \chi_m \rangle\\
&= \hbar^2 \vert u \dd \chi_m \vert^2 + \hbar^2 \langle \chi_m \dd u, u \dd \chi_m \rangle  -  i \hbar \underbrace{ \vert u \vert^2 \langle A, \dd \chi_m \rangle}_{\text{real}}.
\end{align*}
Thus,
\begin{align*}
q_{\hbar}(u) &= \sum_m \left( q_{\hbar}(\chi_m u) - \hbar^2 \Vert u \dd \chi_m \Vert^2 \right) + 2 \hbar^2 \Re \int \sum_m \langle \chi_m \dd u, u \dd \chi_m \rangle \dd q_g\\
&= \sum_m \left( q_{\hbar}(\chi_m u) - \hbar^2 \Vert u \dd \chi_m \Vert^2 \right) +  \hbar^2 \Re \int \sum_m \langle \bar{u} \dd u, 2 \chi_m \dd \chi_m \rangle \dd q_g\\
&= \sum_m \left( q_{\hbar}(\chi_m u) - \hbar^2 \Vert u \dd \chi_m \Vert^2 \right) +  \hbar^2 \Re \int  \langle \bar{u} \dd u, 2 \underbrace{\dd \left( \sum_m \chi_m^2 \right)}_{=0} \rangle \dd q_g\\
&= \sum_m \left( q_{\hbar}(\chi_m u) - \hbar^2 \Vert u \dd \chi_m \Vert^2 \right)
\end{align*}
\end{proof}

\section*{Acknowledgement}

I would like to thank Nicolas Raymond and San Vu Ngoc for many stimulating discussions, helpful advices, and their readings of the preliminary drafts of this version.

\end{document}